\documentclass[preprint,12pt]{elsarticle}
\usepackage[margin=2.5cm]{geometry}

\usepackage{graphicx}
\usepackage{amssymb}

\usepackage[centertags]{amsmath}
\usepackage{amsfonts}
\usepackage{amssymb,bbm}
\usepackage{amsthm}
\usepackage[symbol*]{footmisc}
\usepackage{tikz}
\usepackage{tkz-euclide}
\usepackage[title]{appendix}
\DefineFNsymbolsTM{myfnsymbols}{
  \textasteriskcentered *
  \textdagger    \dagger
  \textdaggerdbl \ddagger
}
\setfnsymbol{myfnsymbols}
\numberwithin{equation}{section}

\newtheorem{thm}{Theorem}[section]
\newtheorem{prop}[thm]{Proposition}

\newtheorem{oss}[thm]{Remark}

\newtheorem{cor}[thm]{Corollary}

\def\RR{{\mathbb{R}}}

\def\NN{{\mathbb{N}}}

\newcommand{\norm}[1]{\left\lVert #1 \right\rVert}

\definecolor{green(ryb)}{rgb}{0.173, 0.627, 0.173}

\journal{Journal Name}

\begin{document}

\begin{frontmatter}


\title{Bilinear control of evolution equations with unbounded lower order terms. Application to the Fokker-Planck equation\tnoteref{fund}}
\tnotetext[fund]{This paper was partly supported by the INdAM National Group for Mathematical Analysis, Probability and their Applications (GNAMPA) and by the French-German-Italian Laboratoire International Associ\'{e} (LIA) COPDESC.}


\author{Fatiha Alabau Boussouira}
\address{Laboratoire Jacques-Louis Lions, Sorbonne Universit\'{e}, 75005, Paris, France

Universit\'{e} de Lorraine, 57000 Metz, France

alabau@ljll.math.upmc.fr}
\author{Piermarco Cannarsa} 
\address{Dipartimento di Matematica, Universit\`{a} di Roma Tor Vergata, 00133, Roma, Italy 

cannarsa@mat.uniroma2.it}
\author{Cristina Urbani}
\address{Dipartimento di Matematica, Universit\`{a} di Roma Tor Vergata, 00133, Roma, Italy 

urbani@mat.uniroma2.it}

\begin{abstract}
We study the exact controllability of the evolution equation
\begin{equation*}
u'(t)+Au(t)+p(t)Bu(t)=0
\end{equation*}
where $A$ is a nonnegative self-adjoint operator on a Hilbert space $X$ and $B$ is an unbounded linear operator on $X$, which is dominated by the square root of $A$. The control action is bilinear and only of scalar-input form, meaning that the control is the scalar function $p$, which is assumed to depend only on time. Furthermore, we only consider square-integrable controls. Our main result is the local exact controllability of the above equation to the ground state solution, that is, the evolution through time, of the first eigenfunction of $A$, as initial data. 

The analogous problem (in a more general form) was addressed in our previous paper [Exact controllablity to eigensolutions for evolution equations of parabolic type via bilinear control, Alabau-Boussouira F., Cannarsa P. and Urbani C., Nonlinear Diff. Eq. Appl. (2022)] for a bounded operator $B$. The current extension to unbounded operators allows for many more applications, including the Fokker-Planck equation in one space dimension, and a larger class of control actions.
\end{abstract}

\end{frontmatter}


\section{Introduction}
In a series of recent papers (\cite{acu,acue,cu}), we have studied stabilization and exact controllability to eigensolutions for evolution equations of the form
\begin{equation}\label{intro-eq-u}
\begin{cases}
u'(t)+Au(t)+p(t)Bu(t)=0,&t\in(0,T)\\
u(0)=u_0
\end{cases}
\end{equation}
where $u_0$ belongs to a Hilbert space $(X,\langle\cdot,\cdot\rangle,\norm{\cdot})$ and
\begin{itemize}
\item[i)] $A:D(A)\subset X\to X$ is a self-adjoint operator with compact resolvent and such that $A\geq -\sigma I$, with $\sigma\geq0$,
\item[ii)] $B:X\to X$ is a bounded linear operator,
\item[iii)] $p\in L^2(0,T)$ is a bilinear control.
\end{itemize}
We refer to equation \eqref{intro-eq-u} as being of parabolic type, since assumption i) is usually satisfied by parabolic operators.

The scalar-input bilinear controllability problem has been addressed by several authors, starting with the negative result by Ball, Marsden, Slemrod \cite{bms}. Controllability issues are interesting also in the hyperbolic or diffusive context, where several results are now available to describe the reachable set of specific partial differential equations in $1$-D, such as the Schr{\"o}dinger equation (\cite{au,beau,bl}) and the classical (\cite{au,b}) and degenerate (\cite{cmu}) wave equation. The above problem enters in the so-called class of nonlinear control problems. 

We refer the readers to the book by Coron\cite{coron} where general control systems are studied as well as mathematical methods to treat them, with a focus on systems for which the nonlinearities are determinant for controllability issues, including the Schr\"odinger equation.

Most of the above results devoted to scalar-input bilinear controllability issues have in common the fact that they address controllability properties of \eqref{intro-eq-u} near the ground state solution $\psi_1(t)=e^{-\lambda_1 t}\varphi_1$ (\cite{b,bl,cmu}) or, more in general, to eigensolutions $\psi_j(t)=e^{-\lambda_j t}\varphi_j$ (\cite{acu,acue}), that are solutions of the free dynamics ($p=0$) associated to \eqref{intro-eq-u}, namely 
\begin{equation*}
u'(t)+Au(t)=0,
\end{equation*}
with initial condition $u_0=\varphi_j$, where we denote by $\{\lambda_k\}_{k\in\NN^*}$ the eigenvalues of $A$, and by $\{\varphi_k\}_{k\in\NN^*}$ the associated eigenfunctions.

To prove local controllability to trajectories, we have introduced in \cite{acu} the notion of $j$-null controllability in time $T>0$ for the pair $\{A,B\}$: we require the existence of a constant $N_T>0$ such that for any initial condition $y_0\in X$, there exists a control $p\in L^2(0,T)$ satisfying
\begin{equation}\label{estimpnew}
\norm{p}_{L^2(0,T)}\leq N_T\norm{y_0} ,
\end{equation}
and for which $y(T)=0$, where $y(\cdot)$ is the solution of the following linear problem
\begin{equation}\label{newlin}
\begin{cases}
y'(t)+Ay(t)+p(t)B\varphi_j=0,&t\in[0,T]\\
y(0)=y_0.
\end{cases}
\end{equation}
We call the best constant, defined as
\begin{equation}\label{controlcost}
N(T):=\sup_{\norm{y_0}=1}\inf\left\{\norm{p}_{L^2(0,T)}:y(T;y_0,p)=0\right\},
\end{equation}
the \emph{control cost}.

Furthermore, we have shown that sufficient conditions for $\{A,B\}$ to be $j$-null controllable are a gap condition of the eigenvalues of $A$ and a rank condition on $B$: we assume $B$ to spread $\varphi_j$ in all directions (see \cite[Theorem 1.2]{acue}).
When, for instance, $X=L^2(0,1)$, $B$ could be taken as a multiplicative operator, i.e.,
\begin{equation*}
(B\varphi)(x)=\mu(x)\varphi(x),\qquad x\in (0,1),
\end{equation*}
where $\mu$ has to be chosen in order to guarantee the dispersive action mentioned above (see \cite{au} for a general method to construct infinite classes of such $\mu$, including polynomial type classes, and for various PDE's, as well as various boundary conditions). 

Another common feature of the aforementioned references is that $B$ is assumed to be bounded. In many applications, however, one is forced to consider a bilinear control acting on a drift term rather than a potential, which leads to allowing $B$ to be unbounded.

A typical example of such a situation occurs when dealing with the Fokker-Planck equation
\begin{equation*}
u_t-u_{xx}-p(t)(\mu(x)u)_x=0,
\end{equation*}
which is satisfied by the probability density of the diffusion process associated with
\begin{equation*}
dX_t=p(t)\mu(X_t)dt+\sqrt{2}dW_t,
\end{equation*}
where $W_t$ is the standard Wiener process on a probability space $(\Omega, \mathcal{A},\mathbb{P})$.

This paper aims to extend the theory of \cite{acue} to unbounded operators $B:D(B)\subset X\to X$ satisfying
\begin{equation*}
D(A_\sigma^{1/2})\hookrightarrow D(B),\quad \text{and}\quad \norm{B\varphi}\leq C\left(\norm{A_\sigma\varphi}^2+\norm{\varphi}^2\right)^{1/2},
\end{equation*}
for some constant $C>0$, where $A_\sigma:=A+\sigma I$ (the aforementioned $j$-null controllability property of the pair $\{A,B\}$ has to be assumed unchanged).

We consider \eqref{intro-eq-u} where $A:D(A)\subset X\to X$ is a self-adjoint accretive operator with compact resolvent and $B$ is an unbounded linear operator. By adapting the approach of \cite{acue}, we first obtain in Section \ref{MainResult} a local controllability result (in the topology of $D(A^{1/2})$), to the first eigensolution $\psi_1$ (the ground state), see Theorem \ref{teo-contr-B-unb}. Then, in Section \ref{SemiGlobal} we derive two semi-global results, Theorem \ref{teoglobal} and \ref{teoglobal0}. Finally, in Section \ref{Applications} we discuss applications to partial differential equations including
\begin{itemize}
\item[a)] the Fokker-Planck equation,
\item[b)] the heat equation with control on the drift  under Neumann boundary conditions,
\item[c)] a class of degenerate parabolic equations under Dirichlet or Neumann boundary conditions.
\end{itemize}

\section{Preliminaries}
Let $(X,\langle\cdot,\cdot\rangle,\norm{\cdot})$ be a separable Hilbert space. Let $A:D(A)\subset X\to X$ be a densely defined linear operator with the following properties:
\begin{equation}\label{ipA}
\begin{array}{ll}
(a) & A \mbox{ is self-adjoint},\\
(b) &\langle Ax,x\rangle \geq0,\,\, \forall\, x\in D(A),\\
(c) &\exists\,\lambda>0\,:\,(\lambda I+A)^{-1}:X\to X \mbox{ is compact}.
\end{array}
\end{equation}
We recall that under the above assumptions $A$ is a closed operator and $D(A)$ is itself a Hilbert space with respect to the scalar product
\begin{equation}
(x,y)_{D(A)}=\langle x,y\rangle+\langle Ax,Ay\rangle,\quad \forall\,x,y\in D(A).
\end{equation}
Moreover, $-A$ is the infinitesimal generator of a strongly continuous semigroup of contractions on $X$ which is also analytic and will be denoted by $e^{-tA}$.

In view of the above assumptions, there exists an orthonormal basis $\{\varphi_k\}_{k\in\NN^*}$ in $X$ of eigenfunctions of $A$, that is, $\varphi_k\in D(A)$ and $A\varphi_k=\lambda_k\varphi_k$ $\forall\, k \in \NN^*$, where $\{\lambda_k\}_{k\in\NN^*}\subset \RR$ denote the corresponding eigenvalues. We recall that $\lambda_k\geq 0$, $\forall\, k\in\NN^*$ and we suppose --- without loss of generality --- that $\{\lambda_k\}_{k\in\NN^*}$ is ordered so that $0\leq\lambda_k\leq\lambda_{k+1}\to \infty$ as $k\to\infty$.
The associated semigroup has the following representation
\begin{equation}\label{semigr}
e^{-tA}\varphi=\sum_{k=1}^\infty\langle \varphi,\varphi_k\rangle e^{-\lambda_k t}\varphi_k,\quad\forall\, \varphi\in X.
\end{equation}
For any $s\geq 0$ we can define the operator $A^s:D(A^s)\subset X\to X$ which is the $s$-fractional power of $A$ (see \cite{p}). Under our assumptions, such a linear operator is characterized as follows
\begin{equation}
\begin{array}{l}
D(A^s)=\left\{x\in X ~\left|~ \sum_{k\in\NN^*}\lambda_k^{2s}|\langle x,\varphi_k\rangle|^2<\infty\right.\right\}\\\\
A^{s} x=\sum_{k\in\NN^*}\lambda_k^{s}\langle x,\varphi_k\rangle\varphi_k,\qquad \forall\, x\in D(A^s).
\end{array}
\end{equation}
The space $D(A^s)$, equipped with the norm
\begin{equation*}
\norm{x}_{D(A^s)}:=\left(\norm{x}^2+\norm{A^sx}^2\right)^{1/2},\quad\forall\, \varphi\in D(A^s),
\end{equation*}
induced by the scalar product
\begin{equation*}
\langle x,y\rangle_s=\langle x,y\rangle+\langle A^{s}x,A^{s}x\rangle
\end{equation*}
is a Hilbert space for any $s\geq0$. Note that we have trivially the inequality
$$
\norm{x}_{D(A^s)} \leqslant \left(\norm{x}+\norm{A^sx}\right),\quad\forall\, \varphi\in D(A^s)
$$
Of course, the right hand side also defines an equivalent norm to $\norm{\cdot}_{D(A^s)}$ on $D(A^s)$, but is not associated to a scalar product. We will make use of the above inequality in some computations without further referring to it.

We indicate by $B_{R,s}(x)$ the unit ball of radius $R$ with respect to the norm $\norm{\cdot}_{D(A^s)}$ centered at $x$.

Let $T>0$ and consider the following problem
\begin{equation}
\begin{cases}\label{controlsystem}
u'(t)+Au(t)=f(t),&t\in[0,T]\\
u(0)=u_0
\end{cases}
\end{equation}
where $u_0\in X$ and $f\in L^2(0,T;X)$. We now recall two definitions of solution of problem \eqref{controlsystem}:
\begin{itemize}
\item the function $u\in C([0,T], X)$ defined by $$u(t)=e^{-tA}u_0+\int_0^t e^{-(t-s)A}f(s)ds$$ is called the \emph{mild solution} of \eqref{controlsystem},
\item a function 
\begin{equation}\label{012}
u\in H^1(0,T;X)\cap L^2(0,T;D(A))
\end{equation}
is called a \emph{strict solution} of \eqref{controlsystem} if $u(0)=u_0$ and $u$ satisfies the equation in \eqref{controlsystem} for a.e. $t\in [0,T]$.
\end{itemize}
The well-posedness of the Cauchy problem \eqref{controlsystem} is a classical result (see, for instance, \cite[Theorem 3.1, p. 143]{bd}). We observe that the space
\begin{equation*}
W(D(A);X)=H^1(0,T;X)\cap L^2(0,T;D(A))
\end{equation*}
is a Banach space with the norm
\begin{equation*}
\norm{\varphi}_W=\left(\norm{\varphi}_{H^1(0,T;X)}^2+\norm{\varphi}_{L^2(0,T;D(A))}^2\right)^{1/2}
\end{equation*}
for all $\varphi\in W(D(A);X)$.

\begin{thm}\label{maxreg}
Let $u_0\in D(A^{1/2})$ and $f\in L^2(0,T;X)$. Under hypothesis \eqref{ipA}, the mild solution of system \eqref{controlsystem} 
\begin{equation}\label{013}
u(t)=e^{-tA}u_0+\int_0^t e^{-(t-s)A}f(t)dt.
\end{equation}
is a strict solution. 

Moreover, there exists a constant $C>0$ such that
\begin{equation}\label{014}
\norm{u}_W\leq C\left(\norm{f}_{L^2(0,T;X)}+\norm{u_0}_{D(A^{1/2})}\right).
\end{equation}
\end{thm}
The regularity \eqref{012} of the function $u$ given by \eqref{013} is called \emph{maximal regularity}. Under our assumptions such a property is due to the analyticity of $e^{-tA}$. Observe that \eqref{014} ensures that the strict solution of \eqref{controlsystem} is unique.

We now recall a useful result which derives from spaces interpolation.
\begin{prop}\label{lio-mag}
Let $u\in W(D(A);X)$ then 
\begin{equation*}
u\in C([0,T];D(A^{1/2})).
\end{equation*}
\end{prop}
We refer to \cite[Proposition 2.1, p. 22 and Theorem 3.1, p. 23]{Lions-Magenes}) for the proof.

From Proposition \ref{lio-mag} we deduce the following regularity property for the solution of problem \eqref{controlsystem}.
\begin{cor}\label{cor-lio-mag}
Let $u_0\in D(A^{1/2})$ and $f\in L^2(0,T;X)$. Then, the strict solution $u$ of \eqref{controlsystem} is such that $u\in C([0,T];D(A^{1/2})$ and there exists $C_0>0$ such that 
\begin{equation}\label{normC0}
\sup_{t\in[0,T]}\norm{u(t)}_{D(A^{1/2})}\leq C_0\left(\norm{f}_{L^2(0,T;X)}+\norm{u_0}_{D(A^{1/2})}\right).
\end{equation}
\end{cor}
Given $T>0$, let $B:D(B)\subset X\to X$ be a linear unbounded operator such that
\begin{equation}\label{ipB}
D(A^{1/2})\hookrightarrow D(B),
\end{equation} 
namely $D(A^{1/2})\subset D(B)$ and there exists a constant $C_B>0$ (that we can suppose, without loss of generality, to be greater than one) such that 
\begin{equation}\label{BA12}
\norm{B\varphi}\leq C_B\norm{\varphi}_{D(A^{1/2})},\qquad \forall\, \varphi\in D(A^{1/2}).
\end{equation}

In the proposition that follows we show the well-posedness of the bilinear control problem with a source term
\begin{equation}\label{a1f}
\begin{cases}
u'(t)+Au(t)+p(t)Bu(t)+f(t)=0,& t\in [0,T]\\
u(0)=u_0.
\end{cases}
\end{equation}
We introduce the following notation: $\forall s\geq0$ we set
\begin{equation*}\begin{array}{l}
\norm{\varphi}_{s}:=\norm{\varphi}_{D(A^s)},\qquad\forall\, \varphi\in D(A^s)\\\\
\norm{f}_{2,s}:=\norm{f}_{L^2(0,T;D(A^s))},\qquad\forall\,f\in L^2(0,T;D(A^s))\\\\
\norm{f}_{\infty,s}:=\norm{f}_{C([0,T];D(A^s))}=\sup_{t\in [0,T]}\norm{A^{s}f(t)},\qquad\forall\, f\in C([0,T];D(A^s)).
\end{array}
\end{equation*}
\begin{prop}\label{propa24}
Let $T>0$ and fix $p\in L^2(0,T)$. Then, for all $u_0\in D(A^{1/2})$ and $f\in L^2(0,T;X)$ there exists a unique mild solution $u$ of \eqref{a1f}.

Moreover, $u\in C([0,T],D(A^{1/2})$ and the following equality holds for every $t\in [0,T]$
\begin{equation}
u(t)=e^{-tA}u_0-\int_0^t e^{-(t-s)A}[p(s)Bu(s)+f(s)]ds.
\end{equation}
Furthermore, $u$ is a strict solution of \eqref{a1f} and there exists a constant $C_1=C_1(p)>0$ such that 
\begin{equation}\label{a5}
||u||_{\infty,1/2}\leq C_1 (||u_0||_{1/2}+||f||_{2,0}).
\end{equation}
\end{prop}

Hereafter, we denote by $C$ a generic positive constant which may differ from line to line. Constants which play a specific role will be distinguished by an index i.e., $C_0$, $C_B$, \dots.

\begin{proof}
The existence and uniqueness of the solution of (\ref{a1f}) comes from a fixed point argument which is based on Theorem \ref{maxreg}.

For any $\xi\in C([0,T];D(A^{1/2})$, let us consider the map 
\begin{equation}\nonumber
\Phi(\xi)(t)=e^{-tA}u_0-\int_0^t e^{-(t-s)A}[p(s)B\xi(s)+f(s)]ds.
\end{equation}
We want to prove that $\Phi$ is a contraction. First, we prove that $\Phi$ maps $C([0,T];D(A^{1/2}))$ into itself. Since $u_0\in D(A^{1/2})$ and $p(\cdot)B\xi(\cdot)$, $f(\cdot)\in L^2(0,T;X)$, applying Theorem \ref{maxreg} and Corollary \ref{cor-lio-mag} it turns out that $\Phi(\xi) \in C([0,T];D(A^{1/2}))$.

Now we have to prove that $\Phi$ is a contraction. Thanks to \eqref{normC0}, we have
\begin{equation}\begin{split}\label{contr}
\norm{\Phi(\xi)-\Phi(\zeta)}_{\infty,1/2}&=\sup_{t\in[0,T]}\norm{\int_0^t e^{-(t-s)A}p(s) B(\xi-\zeta)(s)ds}_{1/2}\\
&\leq C_0\norm{ p B(\xi-\zeta)}_{2,0}\\
&=C_0\left(\int_0^T|p(s)|^2\norm{ B(\xi-\zeta)(s)}^2ds\right)^{1/2}\\
&\leq C_0C_B \left(\int_0^T|p(s)|^2\norm{ (\xi-\zeta)(s)}_{D( A^{1/2})}^2ds\right)^{1/2}\\
&\leq C_0C_B \norm{p}_{L^2(0,T)}\norm{\xi-\zeta}_{\infty,1/2}
\end{split}
\end{equation}
where we have used the fact that $D(A^{1/2}) \hookrightarrow D(B)$.
If $C_0C_B\norm{p}_{L^2(0,T)}<1$, then \eqref{contr} shows that $\Phi$ is a contraction. If this quantity is larger than one, we subdivide the interval $[0,T]$ into $N$ subintervals $[T_0,T_1],[T_1,T_2],\dots,[T_{N-1},T_N]$, with $T_0=0,T_N=T$, such that $C_0C_B\norm{p}_{L^2(T_j,T_{j+1})}\leq 1/2$ in $[T_j,T_{j+1}]$, $\forall\, j=0,\dots, N-1$ and we repeat the contraction argument in each interval.

It remains to prove (\ref{a5}). By using once again \eqref{normC0}, we get
\begin{equation}\begin{split}
\norm{u}_{\infty,1/2}&\leq \norm{u_0}_{1/2}+\sup_{t\in[0,T]}\norm{ \int_0^t e^{-(t-s)A}[p(s) Bu(s)+ f(s)]ds}_{1/2}\\
&\leq \norm{u_0}_{1/2}+C_0(\norm{p Bu}_{2,0}+\norm{f}_{2,0})\\
&= \norm{u_0}_{1/2}+C_0\left( \int_0^T |p(s)|^2\norm{ Bu(s)}^2ds\right)^{1/2}+C_0\norm{f}_{2,0}\\
&\leq \norm{u_0}_{1/2}+C_0C_B\left( \int_0^T |p(s)|^2\norm{ u(s)}_{D(A^{1/2})}^2ds\right)^{1/2}+C_0\norm{f}_{2,0}\\
&= \norm{u_0}_{1/2}+C_0C_B\norm{p}_{L^2(0,T)}\norm{u}_{\infty,1/2}+C_0\norm{f}_{2,0},
\end{split}
\end{equation}
which implies
\begin{equation*}
(1-C_0C_B\norm{p}_{L^2(0,T)})\norm{u}_{\infty,1/2}\leq \norm{u_0}_{1/2}+C_0\norm{f}_{2,0}.
\end{equation*}
If $C_0C_B\norm{p}_{L^2(0,T)}\leq1/2$, then we have inequality (\ref{a5}). Otherwise, to get the conclusion, we proceed subdividing the interval $[0,T]$ into smaller subintervals, as explained above.
\end{proof}

We now consider the following bilinear control problem with a specific source
\begin{equation}\label{v}
\left\{\begin{array}{ll}
v'(t)+A v(t)+p(t)Bv(t)+p(t)B\varphi_1=0,&t\in[0,T]\\\\
v(0)=v_0.
\end{array}\right.
\end{equation}
Let us denote by $v(\cdot;v_0,p)$ the solution of \eqref{v} associated to the initial condition $v_0$ and control $p$. The result that follows provides an estimate of $\sup_{t\in[0,T]}\norm{v(t;v_0,p)}_{1/2}$ in terms of the initial condition $v_0$.

\begin{prop}\label{prop38}
Let $T>0$. Let $A$ and $B$ satisfy hypotheses \eqref{ipA} and \eqref{ipB}, respectively. Let $v_0\in D(A^{1/2})$ and $p\in L^2(0,T)$ be such that 
\begin{equation}\label{NT}
\norm{p}_{L^2(0,T)}\leq N_T\norm{v_0}
\end{equation}
with $N_T$ a positive constant.

Then, the solution of \eqref{v} satisfies
\begin{equation}\label{unifv}
\sup_{t\in[0,T]}\norm{v(t;v_0,p)}^2_{1/2}\leq C_{1,1}(T,\norm{v_0}_{1/2})\norm{v_0}^2_{1/2},
\end{equation}
where 
\begin{multline}\label{C1j}
C_{1,1}(T,\norm{v_0}_{1/2}):=e^{C_BN_T\left(\frac{5}{2}C_BN_T\norm{v_0}_{1/2}+2\sqrt{T}\right)\norm{v_0}_{1/2}+T}\cdot\\
\cdot\left(1+\frac{5}{2}C_B^2(1+\lambda_1^{1/2})^2N_T^2+\frac{3}{2}C^2_BN_T^2\left(C^2_B(1+\lambda_1^{1/2})^2N_T^2+1\right)\norm{v_0}^2_{1/2}\right)
\end{multline}
\end{prop}

\begin{proof}
For the sake of compactness, sometimes we will denote the solution of \eqref{v}, by omitting the reference to the initial condition and the control, as $v(\cdot)$. We perform energy estimates of the equation satisfied by $v$ by taking first the scalar product with $v(t)$
\begin{equation*}
\langle v^\prime(t),v(t) \rangle+\langle Av(t),v(t) \rangle+p(t)\langle Bv(t)+B\varphi_1,v(t) \rangle=0,
\end{equation*}
from which we have that
\begin{equation}\label{barC}
\begin{split}
\frac{1}{2}&\frac{d}{dt}\norm{v(t)}^2+\norm{A^{1/2}v(t)}^{2}\leq |p(t)|\norm{Bv(t)}\norm{v(t)}+|p(t)|\norm{B\varphi_1}\norm{v(t)}\\
&\leq C_B|p(t)|\left(\norm{v(t)}^2+\norm{A^{1/2}v(t)}^2\right)^{1/2}\norm{v(t)}+C_B(1+\lambda_1^{1/2})|p(t)|\norm{v(t)}\\
&\leq C_B|p(t)|\norm{v(t)}^2+\frac{1}{2}\norm{A^{1/2}v(t)}^2+\frac{C^2_B}{2}|p(t)|^2\norm{v(t)}^2+\frac{C^2_B}{2}(1+\lambda_1^{1/2})^2|p(t)|^2+\frac{1}{2}\norm{v(t)}^2.
\end{split}
\end{equation}
Therefore, from the above inequality it follows that
\begin{equation*}
\frac{1}{2}\frac{d}{dt}\norm{v(t)}^2\leq \left(C_B|p(t)|+\frac{C^2_B}{2}|p(t)|^2+\frac{1}{2}\right)\norm{v(t)}^2+\frac{C^2_B}{2}(1+\lambda_1^{1/2})^2|p(t)|^2.
\end{equation*}
We integrate from $0$ to $t$ and thanks to Gronwall's Lemma we deduce that
\begin{equation}\label{supv}
\sup_{t\in[0,T]}\norm{v(t;v_0,p)}^2\leq\left(\norm{v_0}^2+C_B^2(1+\lambda_1^{1/2})^2\norm{p}^2_{L^2(0,T)}\right)e^{C_2(T)},
\end{equation}
where $C_2(T):=2C_B\sqrt{T}\norm{p}_{L^2(0,T)}+C^2_B\norm{p}_{L^2(0,T)}^2+T$.

Now, we multiply the equation in \eqref{v} by $v^\prime(t)$ and we obtain
\begin{equation*}
\langle v'(t),v'(t) \rangle+\langle Av(t),v'(t) \rangle+p(t)\langle Bv(t)+B\varphi_1,v'(t) \rangle=0.
\end{equation*}
We recall that under our assumptions on $A$ the function
\begin{equation*}
t\mapsto \langle Av(t),v(t)\rangle
\end{equation*}
is absolutely continuous on $[0,T]$ and
\begin{equation*}
\frac{d}{dt}\langle Av(t),v(t)\rangle=2\langle Av(t),v'(t)\rangle.
\end{equation*}
Therefore, we have
\begin{equation*}
\begin{split}
\norm{v'(t)}^2+\frac{1}{2}\frac{d}{dt}\norm{A^{1/2}v(t)}^{2}&\leq |p(t)|\norm{Bv(t)}\norm{v'(t)}+|p(t)|\norm{B\varphi_1}\norm{v'(t)}\\
&\leq C_B|p(t)|\left(\norm{v(t)}+\norm{A^{1/2}v(t)}\right)\norm{v'(t)}+C_B|p(t)|(1+\lambda_1^{1/2})\norm{v'(t)}\\
&\leq \frac{3}{4}C^2_B|p(t)|^2\norm{v(t)}^2+\frac{1}{3}\norm{v'(t)}^2+\frac{3}{4}C^2_B|p(t)|^2\norm{A^{1/2}v(t)}^2\\
&\quad+\frac{1}{3}\norm{v'(t)}^2+\frac{3}{4}C^2_B(1+\lambda_1^{1/2})^2|p(t)|^2+\frac{1}{3}\norm{v'(t)}^2,
\end{split}
\end{equation*}
that gives
\begin{equation*}
\frac{d}{dt}\norm{A^{1/2}v(t)}^{2}\leq\left(\frac{3}{2}C^2_B|p(t)|^2\norm{v(t)}^2+\frac{3}{2}C^2_B(1+\lambda_1^{1/2})^2|p(t)|^2\right)+\frac{3}{2}C^2_B|p(t)|^2\norm{A^{1/2}v(t)}^2.
\end{equation*}
By Gronwall's Lemma and using the previous energy estimate \eqref{supv}, we deduce that
\begin{equation*}
\begin{split}
&\norm{A^{1/2}v(t)}^{2}\leq\left(\norm{A^{1/2}v_0}^{2}+\frac{3}{2}C^2_B\int_0^t|p(s)|^2\norm{v(s)}^2ds+\frac{3}{2}C^2_B(1+\lambda_1^{1/2})^2\int_0^t|p(s)|^2ds\right)e^{\frac{3}{2}C^2_B\int_0^t|p(s)|^2ds}\\
&\quad\leq \left(\norm{A^{1/2}v_0}^2+\frac{3}{2}C^2_B\norm{p}^2_{L^2(0,T)}\sup_{t\in[0,T]}\norm{v(t)}^2+\frac{3}{2}C^2_B(1+\lambda_1^{1/2})^2\norm{p}^2_{L^2(0,T)}\right)e^{\frac{3}{2}C^2_B\norm{p}^2_{L^2(0,T)}}\\
&\quad\leq \left(\norm{A^{1/2}v_0}^2+\frac{3}{2}C^2_B\norm{p}^2_{L^2(0,T)}\left(\norm{v_0}^2+C_B^2(1+\lambda_1^{1/2})^2\norm{p}^2_{L^2(0,T)}\right)e^{C_2(T)}\right.\\
&\qquad\qquad\qquad\qquad\qquad\qquad\qquad\qquad\qquad\qquad\qquad\left.+\frac{3}{2}C^2_B(1+\lambda_1^{1/2})^2\norm{p}^2_{L^2(0,T)}\right)e^{\frac{3}{2}C^2_B\norm{p}^2_{L^2(0,T)}}\\
&\quad\leq\left(\norm{A^{1/2}v_0}^2+\frac{3}{2}C^2_B\norm{p}^2_{L^2(0,T)}\norm{v_0}^2+\frac{3}{2}C^4_B(1+\lambda_1^{1/2})^2\norm{p}^4_{L^2(0,T)}\right.\\
&\qquad\qquad\qquad\qquad\qquad\qquad\qquad\qquad\qquad\qquad\qquad\left.+\frac{3}{2}C^2_B(1+\lambda_1^{1/2})^2\norm{p}^2_{L^2(0,T)}\right)e^{C_3(T)}
\end{split}
\end{equation*}
with $C_3(T):=\frac{3}{2}C^2_B\norm{p}^2_{L^2(0,T)}+C_2(T)$.

Taking the supremum over the interval $[0,T]$, we have that
\begin{multline}\label{supA1/2v}
\sup_{t\in[0,T]}\norm{A^{1/2}v(t)}^{2}\leq\left(\norm{A^{1/2}v_0}^2+\frac{3}{2}C^2_B\norm{p}^2_{L^2(0,T)}\norm{v_0}^2+\frac{3}{2}C^4_B(1+\lambda_1^{1/2})^2\norm{p}^4_{L^2(0,T)}\right.\\
\left.+\frac{3}{2}C^2_B(1+\lambda_1^{1/2})^2\norm{p}^2_{L^2(0,T)}\right)e^{C_3(T)}.
\end{multline}
Finally, combining \eqref{supv} and \eqref{supA1/2v}, we find that
\begin{equation*}
\begin{split}
\sup_{t\in[0,T]}\norm{v(t)}^2_{1/2}&\leq \sup_{t\in[0,T]}\norm{v(t)}^2+\sup_{t\in[0,T]}\norm{A^{1/2}v(t)}^2\\
&\leq e^{C_3(T)}\left(\norm{v_0}^2_{1/2}+\frac{5}{2}C_B^2(1+\lambda_1^{1/2})^2\norm{p}^2_{L^2(0,T)}+\frac{3}{2}C^2_B\norm{p}^2_{L^2(0,T)}\norm{v_0}^2\right.\\
&\,\,\,\qquad\qquad\qquad\qquad\qquad\qquad\qquad\qquad\qquad\qquad\left.+\frac{3}{2}C^4_B(1+\lambda_1^{1/2})^2\norm{p}^4_{L^2(0,T)}\right)
\end{split}
\end{equation*}
and using estimate \eqref{NT} we conclude that
\begin{equation*}
\sup_{t\in[0,T]}\norm{v(t)}^2_{1/2}\leq e^{C_3(T)}\left(1+\frac{5}{2}C_B^2(1+\lambda_1^{1/2})^2N_T^2+\frac{3}{2}C^2_BN_T^2\left(C^2_B(1+\lambda_1^{1/2})^2N_T^2+1\right)\norm{v_0}^2_{1/2}\right)\norm{v_0}^2_{1/2}.
\end{equation*}
Thus, we get \eqref{unifv} with $C_{1,1}(T,\norm{v_0})$ defined in \eqref{C1j}.
\end{proof}

We now turn our attention to the following control problem
\begin{equation}\label{w}
\left\{\begin{array}{ll}
w'(t)+Aw(t)+p(t)Bv(t)=0,&t\in[0,T]\\\\
w(0)=0
\end{array}\right.
\end{equation}
where $v$ solves \eqref{v}. We will denote by $w(\cdot;0,p)$ the solution of \eqref{w}. In the result that follows we give a quadratic estimate of $w(\cdot;0,p)$ in terms of the initial condition of the problem solved by $v$.

\begin{prop}\label{prop39}
Let $T>0$ and let $A$ and $B$ satisfy hypotheses \eqref{ipA} and \eqref{ipB}, respectively. Let $p\in L^2(0,T)$ satisfy \eqref{NT} with $N_T=:N(T)$ and $v_0\in D(A^{1/2})$ be such that
\begin{equation}\label{v0}
N(T)\norm{v_0}_{1/2}\leq 1.
\end{equation}
Then, it holds that
\begin{equation}\label{wT}
\norm{w(T;0,p)}_{1/2}\leq K(T)\norm{v_0}_{1/2}^2,
\end{equation}
where
\begin{equation}\label{K}
K(T)^2:=2e^{C_4(T)}C^2_BN(T)^2C_5(T)
\end{equation}
with
\begin{equation*}
C_4(T):=C_B\left(\frac{5}{2}C_B+2\sqrt{T}\right)+2T,
\end{equation*}
\begin{equation*}
C_5(T)=1+\frac{5}{2}C^2_B(1+\lambda_1^{1/2})^2N(T)^2+\frac{3}{2}C^2_B\left(C^2_B(1+\lambda_1^{1/2})^2N(T)^2+1\right).
\end{equation*}
\end{prop}

\begin{proof} 
For the sake of compactness, sometimes we will denote the solution of \eqref{w} by $w(\cdot)$. First, we multiply the equation in \eqref{w} by $w(t)$, and we get
\begin{equation*}
\langle w'(t),w(t)\rangle+\langle Aw(t),w(t)\rangle+p(t)\langle Bv(t),w(t)\rangle=0,
\end{equation*}
that implies
\begin{equation*}
\begin{split}
\frac{1}{2}\frac{d}{dt}\norm{w(t)}^2&\leq|p(t)|\norm{Bv(t)}\norm{w(t)}\\
&\leq C_B|p(t)|\norm{v(t)}_{1/2}\norm{w(t)}\\
&\leq \frac{C^2_B}{2}|p(t)|^2\norm{v(t)}^2_{1/2}+\frac{1}{2}\norm{w(t)}^2.
\end{split}
\end{equation*}
Applying Gronwall's Lemma we obtain
\begin{equation*}
\norm{w(t)}^2\leq C^2_Be^t\int_0^t|p(s)|^2\norm{v(s)}_{1/2}^2ds,
\end{equation*}
and therefore
\begin{equation}\label{supw}
\sup_{t\in[0,T]}\norm{w(t)}^2\leq C_B^2e^T\norm{p}^2_{L^2(0,T)}\sup_{t\in[0,T]}\norm{v(t)}_{1/2}^2.
\end{equation}

Now, we multiply the equation in \eqref{w} by $w^\prime(t)$,
\begin{equation*}
\langle w'(t),w'(t)\rangle+\langle Aw(t),w'(t)\rangle+p(t)\langle Bv(t),w'(t)\rangle=0,
\end{equation*}
from which we deduce
\begin{equation*}
\begin{split}
\norm{w'(t)}^2+\frac{1}{2}\frac{d}{dt}\norm{A^{1/2}w(t)}^2&\leq |p(t)|\norm{Bv(t)}\norm{w'(t)}\\
&\leq C_B|p(t)|\norm{v(t)}_{1/2}\norm{w'(t)}\\
&\leq \frac{C^2_B}{2}|p(t)|^2\norm{v(t)}_{1/2}^2+\frac{1}{2}\norm{w'(t)}^2.
\end{split}
\end{equation*}
Therefore, it holds that
\begin{equation*}
\frac{d}{dt}\norm{A^{1/2}w(t)}^2\leq C^2_B|p(t)|^2\norm{v(t)}_{1/2}^2,
\end{equation*}
that yields the following estimate
\begin{equation}\label{supA1/2w}
\sup_{t\in[0,T]}\norm{A^{1/2}w(t)}^2\leq C^2_B\norm{p}^2_{L^2(0,T)}\sup_{t\in[0,T]}\norm{v(t)}_{1/2}^2.
\end{equation}

Combining \eqref{supw} and \eqref{supA1/2w} we have
\begin{equation*}
\begin{split}
\sup_{t\in[0,T]}\norm{w(t)}^2_{1/2}&\leq \sup_{t\in[0,T]}\norm{w(t)}^2+\sup_{t\in[0,T]}\norm{A^{1/2}w(t)}^2\\
&\leq 2e^TC^2_B\norm{p}^2_{L^2(0,T)}\sup_{t\in[0,T]}\norm{v(t)}_{1/2}^2
\end{split}
\end{equation*}
and thanks to the estimates of $\sup_{t\in[0,T]}\norm{v(t)}_{1/2}^2$ given by \eqref{unifv} and of $\norm{p}_{L^2(0,T)}$ given by \eqref{NT}, we deduce that
\begin{equation}
\sup_{t\in[0,T]}\norm{w(t)}^2_{1/2}\leq 2e^TC^2_BC_{1,1}(T,\norm{v_0}_{1/2})N(T)^2\norm{v_0}^4_{1/2}.
\end{equation}
Finally, thanks to hypothesis \eqref{v0}, we obtain the claim.
\end{proof}

\section{Main result}\label{MainResult}
Let $T>0$. In a separable Hilbert space $(X,\langle\cdot,\cdot\rangle,\norm{\cdot})$, we consider the control problem
\begin{equation}\label{u}
\begin{cases}
u'(t)+Au(t)+p(t)Bu(t)=0,& t\in [0,T]\\
u(0)=u_0
\end{cases}
\end{equation}
where $A:D(A)\subset X\to X$ is a densely defined linear operator satisfying \eqref{ipA}, $B:D(B)\subset X\to X$ is an unbounded linear operator which satisfies \eqref{ipB} and $p\in L^2(0,T)$ is a bilinear control. We denote by $u(\cdot;u_0,p)$ the solution of \eqref{u} associated to the initial condition $u_0$ and control $p$. We call the ``ground state solution" of problem \eqref{u} the function $\psi_1(t)=e^{-\lambda_1 t}\varphi_1$, where $\lambda_1\geq0$ is the first eigenvalue of $A$ and $\varphi_1$ is the associated eigenfunction.

For any $0\leq s_0\leq s_1$ we also introduce the linear control problem
\begin{equation}\label{ys0s1}
\begin{cases}
y'(t)+Ay(t)+p(t)B\varphi_1=0&t\in[s_0,s_1]\\
y(s_0)=y_0
\end{cases}
\end{equation}
and we denote by $y(\cdot;y_0,s_0,p)$ its solution associated to the initial condition $y_0$, attained at time $s_0$, and control $p$.

We recall that the pair $\{A,B\}$ is called $1$-null controllable in time $T$ if there exists a constant $N_T>0$ such that for any $y_0\in X$ there exists a control $p\in L^2(0,T)$ with
\begin{equation*}
\norm{p}_{L^2(0,T)}\leq N_T\norm{y_0}
\end{equation*}
such that $y(T;y_0,0,p)=0$. If $\{A,B\}$ is $1$-null controllable in time $T$, we call the constant $N(T)$ defined in \eqref{controlcost} the control cost.

We now state our main controllability result.
\begin{thm}\label{teo-contr-B-unb}
Let $A:D(A)\subset X\to X$ be a densely defined linear operator that satisfies \eqref{ipA}. Let $B:D(B)\subset X\to X$ be a linear unbounded operator such that \eqref{ipB} holds. Let $\{A,B\}$ be $1$-null controllable in any $T>0$ with control cost $N(\cdot)$ such that there exist $\nu,T_0>0$ for which
\begin{equation}\label{bound-control-cost}
N(\tau)\leq e^{\nu/\tau},\quad\forall\,0<\tau\leq T_0.
\end{equation}

Then, for any $T>0$, there exists a constant $R_{T}>0$ such that, for any $u_0\in B_{R_{T}, 1/2}(\varphi_1)$, there exists a control $p\in L^2(0,T)$ for which system \eqref{u} is locally controllable to the the ground state solution in time $T$, that is, $u(T;u_0,p)=\psi_1(T)$.

Moreover, the following estimate holds
\begin{equation}\label{boundL2norm-p}
\norm{p}_{L^2(0,T)}\leq \frac{e^{-\pi^2\Gamma_0/T}}{e^{2\pi^2\Gamma_0/(3T)}-1},
\end{equation}
where
\begin{equation}\label{Gamma0}
\Gamma_0:=2\nu+\max\{\log D, 0\}
\end{equation}
\begin{equation*}
R_T:=e^{-6\Gamma_0/T_1}
\end{equation*}
with
\begin{equation*}
D:=2\sqrt{2}C_Be^{C_B\left(\frac{5}{4}C_B+1\right)+1}\left(\max\left\{1+\frac{3}{2}C^2_B,\frac{C^2_B}{2}(1+\lambda_1^{1/2})^2(5+3C^2_B)\right\}\right)^{1/2}
\end{equation*}
\begin{equation}\label{Talpha}
T_1=\min\{\frac{6}{\pi^2}T,1,T_0\}.
\end{equation}
\end{thm}

\begin{oss}
\emph{Let us notice that if $A:D(A)\subseteq X\to X$ satisfies
\begin{equation*}
\exists\,\sigma>0\,:\,\langle Ax,x\rangle\geq-\sigma\norm{x}^2,\,\forall\,x\in D(A)
\end{equation*}
instead of item (b) of \eqref{ipA}, and
\begin{equation*}
\exists\,\lambda>\sigma\,:\, (\lambda I+A)^{-1}:X\to X\text{ is compact}
\end{equation*}
instead of (c), then it is always possible to apply Theorem \ref{teo-contr-B-unb}, and prove the controllability of $u$ to the associated ground state solution $\psi_1(t)=e^{\sigma t}\varphi_1$ in any positive time $T>0$, if $D((A+\sigma I)^\frac{1}{2})\hookrightarrow D(B)$. }

\emph{Indeed, by the change of variable $z(t;u_0,p)=e^{-\sigma t}u(t;u_0,p)$, one has that $z$ solves the following problem
\begin{equation}\label{pb-Asigma}
\begin{cases}
z'(t)+A_\sigma z(t)+p(t)Bz(t)=0\\
z(0)=u_0
\end{cases}
\end{equation}
where $A_\sigma=A+\sigma I$ satisfies \eqref{ipA}, and where $\sigma=-\lambda_1$}. 

\emph{We can then apply Theorem \eqref{teo-contr-B-unb} and deduce that the solution of \eqref{pb-Asigma} is controllable in time $T$ to the first eigensolution $\psi_1^\sigma(t)\equiv\varphi_1$. Finally, from this latter result we get
\begin{equation*}
u(T)-\psi_1(T)=e^{\sigma T}z(T)-e^{\sigma T}\varphi_1=e^{\sigma T}(z(T)-\varphi_1)=0
\end{equation*}
which implies the controllability of $u$ to the ground state solution $\psi_1$.}
\end{oss} 

\subsection{Proof of Theorem \ref{teo-contr-B-unb}}

Our aim is to show the controllability of system \eqref{u} to the ground state solution $\psi_1(t)=e^{-\lambda_1 t}\varphi_1$, that is the solution of \eqref{u} when $p=0$ and $u_0=\varphi_1$. We recall that $\{\varphi_k\}_{k\in\NN^*}$ is a basis of $X$ of orthonormal eigenfunctions of the operator $A$, associated to the eigenvalues $\{\lambda_k\}_{k\in\NN^*}$: $A\varphi_k=\lambda_k\varphi_k$ for all $k\in\NN^*$. The proof of Theorem \ref{teo-contr-B-unb} is divided into two parts: we first consider the case $\lambda_1=0$ and prove the controllability result to the corresponding stationary eigensolution $\psi_1(t)\equiv\varphi_1$. Then, we recover the result for the case $\lambda_1>0$ in the second part.

\subsubsection{Case $\lambda_1=0$}

We define the constant
\begin{equation}\label{Tf}
T_f:=\min\{T,\frac{\pi^2}{6},\frac{\pi^2}{6}T_0\},
\end{equation}
where $T>0$ and $T_0$ is defined in \eqref{bound-control-cost}. In what follows we construct a control $p\in L^2(0,T_f)$ which drives the solution of \eqref{u} to $\psi_1$ in time $T_f$.

Set 
\begin{equation}\label{T1}
T_1:=\frac{6}{\pi^2}T_f.
\end{equation}
It is easy to see that $0<T_1\leq 1$. We now define the sequence $\{T_j\}_{j\in\NN^*}$ by
\begin{equation}\label{Tj}
T_j:=\frac{T_1}{j^2},
\end{equation}
and the time steps
\begin{equation}\label{taun}
\tau_n=\sum_{j=1}^n T_j,\qquad\forall\, n\in\NN,
\end{equation}
with the convention that $\sum_{j=1}^0T_j=0$. Notice that $\sum_{j=1}^\infty T_j=\frac{\pi^2}{6}T_1=T_f$.

Set $v:=u-\varphi_1$. Then, for any $0\leq s_0\leq s_1$, $v$ solves the following bilinear control problem
\begin{equation}\label{eq-v}
\begin{cases}
v'(t)+Av(t)+p(t)Bv(t)+p(t)B\varphi_1=0&t\in[s_0,s_1]\\
v(s_0)=v_0
\end{cases}
\end{equation}
with $v_0=u(s_0)-\varphi_1$. We denote by $v(\cdot; v_0,s_0,p)$ the solution of \eqref{eq-v} associated to the initial condition $v_0$ attained at $s_0$ and control $p$. Observe that showing the controllability of $u$ to $\varphi_1$ in time $T_f$ is equivalent to prove null controllabiliy for the solution $v$ of \eqref{eq-v}: $v(T_f;u_0-\varphi_1,0,p)=0$. 

We follow the strategy of the proof of \cite[Theorem 1.1]{acue} which consists first of obtaining an estimate of the solution of \eqref{eq-v} at $T_1$ by the square of the norm of the initial condition thanks to the construction of a suitable control $p_1\in L^2(0,T_1)$. Then, the same procedure is iterated in consecutive time steps $[\tau_{n-1},\tau_n]$ in which we build a control $p_n\in L^2(\tau_{n-1},\tau_n)$ such that, setting
\begin{equation}\label{qn-vn}
\begin{array}{ll}
q_n:=\sum_{j=1}^np_j(t)\chi_{[\tau_{n-1},\tau_n]}(t),\\
v_n:=v(\tau_n;v_0,0,q_n),
\end{array}
\end{equation}
it holds that
\begin{equation}\label{properties}
\begin{array}{ll}
1.& \norm{p_n}_{L^2(\tau_{n-1},\tau_n)}\leq N(T_n)\norm{v_{n-1}},\\
2.& y(\tau_n;v_{n-1},\tau_{n-1},p_n)=0,\\
3.& \norm{v(\tau_n;v_{n-1},\tau_{n-1},p_n)}_{1/2}\leq e^{(\sum_{j=1}^n 2^{n-j}j^2-2^n 6)\Gamma_0/T_1},\\
4.& \norm{v(\tau_n;v_{n-1},\tau_{n-1},p_n)}_{1/2}\leq\prod_{j=1}^nK(T_j)^{2^{n-j}}\norm{v_0}^{2^n}_{1/2},
\end{array}
\end{equation}
where $y(\cdot;v_{n-1},\tau_{n-1},p_n)$ is the solution of \eqref{ys0s1} in $[\tau_{n-1},\tau_n]$ with initial condition $v_{n-1}$ and control $p_n$. We recall that $K(\cdot)$ is defined in \eqref{K}.

\textbf{First step}: we consider problem \eqref{eq-v} with $[s_0,s_1]=[0,T_1]$. Since by hypothesis the pair $\{A,B\}$ is $1$-null controllable in any time $T>0$, then for any $v_0\in D(A^{1/2})$ there exists a control $p_1\in L^2(0,T_1)$ such that
\begin{equation*}
\norm{p_1}_{L^2(0,T_1)}\leq N(T_1)\norm{v_0}\quad\text{and}\quad y(T_1;v_0,0,p_1)=0
\end{equation*}
with $N(\cdot)$ the control cost defined in \eqref{bound-control-cost} and $y(\cdot;v_0,0,p_1)$ the solution of \eqref{ys0s1} in $[s_0,s_1]=[0,T_1]$. Therefore, 1. and 2. of \eqref{properties} are satisfied. We apply now Proposition \ref{prop38} with $N_T:=N(T_1)$ and we obtain 
\begin{equation*}
\sup_{t\in[0,T_1]}\norm{v(t)}^2_{1/2}\leq C_{1,1}(T_1,\norm{v_0}_{1/2})\norm{v_0}^2_{1/2}
\end{equation*}
with $C_{1,1}$ defined in \eqref{C1j}.

In order to prove 3. and 4. of \eqref{properties} we introduce the function $w(t):=v(t;v_0,0,p_1)-y(t;v_0,0,p_1)$. It is easy to see that $w$ solves \eqref{w} with $T=T_1$ and $p=p_1$. We then apply Proposition \ref{prop39} and we deduce that if
\begin{equation}\label{condition-on-v_0}
N(T_1)\norm{v_0}_{1/2}\leq 1
\end{equation}
then
\begin{equation}\label{estimate-w}
\norm{w(T_1;0,p_1)}_{1/2}=\norm{v(T_1;v_0,0,p_1)}_{1/2}\leq K(T_1)\norm{v_0}^2_{1/2}
\end{equation}
where $K(\cdot)$ is defined in \eqref{K}. Observe that, thanks to the definition of $T_1$ and to \eqref{bound-control-cost}, we infer that there exists a constant $\Gamma_0>\nu$ such that
\begin{equation}\label{boundK}
K(\tau)\leq e^{\Gamma_0/\tau},\quad 0<\tau\leq T_1.
\end{equation}
Observe that, a possible choice of $\Gamma_0$ is given by \eqref{Gamma0}.

Let $R_T=e^{-6\Gamma_0/T_1}$ with $T_1$ defined as in \eqref{T1}. Let us prove that if $v_0\in B_{R_T,1/2}(0)$ then \eqref{condition-on-v_0} is satisfied:
\begin{equation*}
N(T_1)\norm{v_0}\leq N(T_1)\norm{v_0}_{1/2}\leq e^{\nu/T_1}e^{-6\Gamma_0/T_1}\leq e^{-5\Gamma_0/T_1}\leq 1
\end{equation*}
where we have used that $\Gamma_0>\nu$, so that \eqref{estimate-w} holds.

Thus, from \eqref{estimate-w} we deduce that
\begin{equation*}
\norm{v(T_1;v_0;0,p_1)}_{1/2}\leq K(T_1)\norm{v_{0}}^2_{1/2}\leq e^{\Gamma_0/T_1}e^{-12\Gamma_0/T_1}=e^{-11\Gamma_0/T_1}
\end{equation*}
which proves 3. and 4. of \eqref{properties}.

\textbf{Iterative step}: for the sake of completeness we report the proof of \cite[Section 3.1.2]{acue} adapted to the current choice of functional setting. 

Suppose that for every $j=1,\dots,n-1$ we have built controls $p_j\in L^2(\tau_{j-1},\tau_j)$ such that \eqref{properties} is satisfied. At the step $n-1$ we have construct $p_{n-1}\in L^2(\tau_{n-2},\tau_{n-1})$ such that
\begin{equation}\label{propertiesn-1}
\begin{array}{ll}
1.& \norm{p_{n-1}}_{L^2(\tau_{n-2},\tau_{n-1})}\leq N(T_{n-1})\norm{v_{n-2}},\\
2.& y(\tau_{n-1};v_{n-2},\tau_{n-2},p_{n-1})=0,\\
3.& \norm{v(\tau_{n-1};v_{n-2},\tau_{n-2},p_{n-1})}_{1/2}\leq e^{(\sum_{j=1}^{n-1} 2^{n-1-j}j^2-2^{n-1} 6)\Gamma_0/T_1},\\
4.& \norm{v(\tau_{n-1};v_{n-2},\tau_{n-2},p_{n-1})}_{1/2}\leq\prod_{j=1}^{n-1}K(T_j)^{2^{n-1-j}}\norm{v_0}^{2^{n-1}}_{1/2}.
\end{array}
\end{equation}
Let us prove that there exists $p_n\in L^2(\tau_{n-1},\tau_n)$ which verifies \eqref{properties}. First, we define $q_{n-1}$ and $v_{n-1}$ as in \eqref{qn-vn}. Consider then the problem
\begin{equation}\label{vn-1}
\begin{cases}
v'(t)+Av(t)+p(t)Bv(t)+p(t)B\varphi_1=0&t\in[\tau_{n-1},\tau_n]\\
v(\tau_{n-1})=v_{n-1}
\end{cases}
\end{equation}
where control $p$ has to be properly chosen. We apply the change of variables $s=t-\tau_{n-1}$ so that we shift the problem into the interval $[0,T_n]$. By introducing the new variables $\tilde{v}(s)=v(s+\tau_{n-1})$ and $\tilde{p}(t)=p(s+\tau_{n-1})$, problem \eqref{vn-1} becomes
\begin{equation}\label{tildevn-1}
\begin{cases}
\tilde{v}'(t)+A\tilde{v}(t)+p(t)B\tilde{v}(t)+\tilde{p}(t)B\varphi_1=0&t\in[0,T_n]\\
\tilde{v}(0)=v_{n-1}.
\end{cases}
\end{equation}
Since $\{A,B\}$ is $1$-null controllable in any positive time, there exists $\tilde{p}_n\in L^2(0,T_n)$ such that
\begin{equation*}
\norm{\tilde{p}_n}_{L^2(0,T_n)}\leq N(T_n)\norm{v_{n-1}}\quad\text{and}\quad \tilde{y}(T_n;v_{n-1},0,\tilde{p}_n)=0
\end{equation*}
with $\tilde{y}(\cdot;v_{n-1},0,\tilde{p}_n)$ solves \eqref{eq-v} on $[0,T_n]$. Moreover, since $v_{n-1}=v(\tau_{n-1};v_0,0,q_{n-1})=v(\tau_{n-1};v_{n-2},\tau_{n-2},p_{n-1})$, we deduce that
\begin{equation}\label{NTnvn-1}
\begin{split}
N(T_n)\norm{v_{n-1}}_{1/2}&\leq e^{\nu n^2/T_1}e^{\left(\sum_{j=1}^{n-1}2^{n-1-j}j^2-2^{n-1}6\right)\Gamma_0/T_1}\\
&\leq e^{\left(n^2+(-(n-1)^2-4(n-1)+2^{n-1}6-6-2^{n-1}6\right)\Gamma_0/T_1}\\
&=e^{-\left(2n+3\right)\Gamma_0/T_1}\leq1
\end{split}
\end{equation}
from 3. of \eqref{propertiesn-1}. Observe that we have used that $\nu\leq \Gamma_0$ and the identity
\begin{equation*}
\sum_{j=0}^n\frac{j^2}{2^j}=2^{-n}\left(-n^2-4n+6(2^n-1)\right),\quad n\geq0.
\end{equation*}
We now choose $\tilde{p}=\tilde{p}_n$ in \eqref{tildevn-1} and we keep denoting by $\tilde{v}$ the corresponding solution. Define $w=\tilde{v}-\tilde{y}$ and observe that $w$ solves \eqref{w} with $v=\tilde{v}$, $T=T_n$ and control $p=\tilde{p}_n$. Thanks to \eqref{NTnvn-1}, we can apply Proposition \ref{prop39} with $T=T_n$ and we deduce that
\begin{equation*}
\norm{w(T_n;0,\tilde{p}_n)}_{1/2}=\norm{\tilde{v}(T_n;v_{n-1},0,\tilde{p}_n)}_{1/2}\leq K(T_n)\norm{v_{n-1}}_{1/2}^2.
\end{equation*}
We define $p_n(t):=\tilde{p}_n(t-\tau_{n-1})$. By shifting back the problem into the interval $[\tau_{n-1},\tau_n]$ we obtain
\begin{equation*}
\norm{p_n}_{L^2(\tau_{n-1},\tau_n)}\leq N(T_n)\norm{v_{n-1}}_{1/2}\quad\text{and}\quad y(\tau_n;v_{n-1},\tau_{n-1},p_n)=0
\end{equation*}
and 
\begin{equation}\label{vnKvn-1}
\norm{v(\tau_n;v_{n-1},\tau_{n-1},p_n)}_{1/2}\leq K(T_n)\norm{v_{n-1}}^2_{1/2}.
\end{equation}
Thus, the first two items of \eqref{properties} are proved. Using 3. of \eqref{propertiesn-1} and \eqref{boundK}, we deduce that
\begin{equation*}
\begin{split}
\norm{v(\tau_n;v_{n-1},\tau_{n-1},p_n)}_{1/2}&\leq e^{\Gamma_0 n^2/T_1}\left[e^{\left(\sum_{j=1}^{n-1}2^{n-1-j}j^2-2^{n-1}6\right)\Gamma_0/T_1}\right]^2\\
&=e^{\left(\sum_{j=1}^n2^{n-j}j^2-2^n6\right)\Gamma_0/T_1}.
\end{split}
\end{equation*}
Therefore, item 3. of \eqref{properties} is verified. Finally, using again \eqref{vnKvn-1} and 4. of \eqref{propertiesn-1} we get
\begin{equation*}
\begin{split}
\norm{v(\tau_n;v_{n-1},\tau_{n-1},p_n)}_{1/2}&\leq K(T_n)\left[\prod_{j=1}^{n-1}K(Tj)^{2^{n-1-j}}\norm{v_0}^{2^{n-1}}_{1/2}\right]^{2}\\
&=\prod_{j=1}^nK(T_j)^{2^{n-j}}\norm{v_0}^{2^n}_{1/2}.
\end{split}
\end{equation*}
The induction argument is then concluded.

We can now complete the proof of our theorem for the case $\lambda_1=0$. Notice that for all $n\in\NN^*$ we have that
\begin{equation*}
\begin{split}
\norm{v(\tau_n;v_{n-1},\tau_{n-1},p_n)}_{1/2}&\leq \prod_{j=1}^nK(T_j)^{2^{n-j}}\norm{v_0}^{2^n}_{1/2}\leq \prod_{j=1}^n\left(e^{\Gamma_0j^2/T_1}\right)^{2^{n-j}}\norm{v_0}^{2^n}_{1/2}\\
&\leq e^{\Gamma_02^n/T_1\sum_{j=1}^n j^2/2^j}\norm{v_0}^{2^n}_{1/2}\leq \left(e^{6\Gamma_0/T_1}\norm{v_0}_{1/2}\right)^{2^n}.
\end{split}
\end{equation*}
The above inequality is equivalent to
\begin{equation*}
\norm{v(\tau_n;v_{0},0,q_n)}_{1/2}\leq\left(e^{6\Gamma_0/T_1}\norm{v_0}_{1/2}\right)^{2^n}
\end{equation*}
where $q_n$ is defined in \eqref{qn-vn}.

Taking the limit as $n\to +\infty$ we deduce that
\begin{equation*}
\norm{u(T_f;u_0,q_\infty)-\varphi_1}_{1/2}=\norm{v(T_f;v_0,0,q_\infty)}_{1/2}\leq0
\end{equation*}
where $T_f$ is defined in \eqref{Tf}. Indeed, by hypothesis $u_0\in B_{R_T,1/2}(\varphi_1)$, with $R_T$ defined in \eqref{Gamma0}, and so $\norm{v_0}<e^{-6\Gamma_0/T_1}$. This means that we have construct a control $p\in L^2_{loc}([0,+\infty))$, defined as follows
\begin{equation*}
p(t)=\begin{cases}
\sum_{n=1}^\infty p_n(t)\chi_{[\tau_{n-1},\tau_n]}&t\in(0,T_f]\\
0&t>T_f,
\end{cases}
\end{equation*}
that steers the solution $u$ of \eqref{u} to the ground state solution in time $T_f$, less or equal to $T$.

Moreover, we can give an upper bound of the $L^2$-norm of the control:
\begin{equation*}
\begin{split}
\norm{p}^2_{L^2(0,T)}&=\sum_{n=1}^\infty\norm{p_n}^2_{L^2(\tau_{n-1},\tau_n)}\leq \sum_{n=1}^\infty\left(N(T_n)\norm{v_{n-1}}_{1/2}\right)^2\\
&\leq \sum_{n=1}^\infty e^{-2(2n+3)\Gamma_0/T_1}\leq\frac{e^{-6\Gamma_0/T_1}}{e^{4\Gamma_0/T_1}-1}=\frac{e^{-\pi^2\Gamma_0/T_f}}{e^{2\pi^2\Gamma_0/3T_f}-1}
\end{split}
\end{equation*}
where we have used 1. of \eqref{properties} and \eqref{NTnvn-1}. Recalling that $T_f\leq T$, we obtain \eqref{boundL2norm-p}.

\subsubsection{Case $\lambda_1>0$}

If $\lambda_1>0$ the result is easily deducible from the previous case. Indeed, we introduce the operator 
\begin{equation*}
A_1:=A-\lambda_1 I,
\end{equation*}
Observe that
\begin{itemize}
\item $A_1$ satisfies \eqref{ipA},
\item $A_1$ has the same eigenfunctions of $A$, $\{\varphi_j\}_{j\in\NN^*}$, and the first eigenvalues of $A_1$ is equal to $\mu_1=\lambda_1-\lambda_1=0$,
\item $\{A_1,B\}$ is $1$-null controllable with associated control cost $N_1(\cdot)$ that satisfies \eqref{bound-control-cost}.
\end{itemize}
Hence, once proved the exact controllability in time $T$ of the following problem
\begin{equation*}
\left\{\begin{array}{ll}
z'(t)+A_1u(t)+p(t)Bz(t)=0,& t\in [0,T]\\\\
z(0)=u_0
\end{array}\right.
\end{equation*} 
to the associated ground state solution $\tilde{\psi_1}=e^{-\mu_1 t}\varphi_1=\varphi_1$, we introduce the function $u(t):=z(t)e^{-\lambda_1 t}$ that is the solution of
\begin{equation}\label{eq-u-theo}
\left\{\begin{array}{ll}
u'(t)+Au(t)+p(t)Bu(t)=0,& t\in [0,T]\\\\
u(0)=u_0
\end{array}\right.
\end{equation} 
and satisfies
\begin{equation*}
\norm{u\left(T\right)-\psi_1\left(T\right)}=\norm{e^{-\lambda_1T}z\left(T\right)-e^{-\lambda_1T}\varphi_1}=e^{-\lambda_1T}\norm{z\left(T\right)-\varphi_1}=0.
\end{equation*}
Therefore, we have shown that \eqref{eq-u-theo} is exacly controllable to the ground state solution $\psi_1(t)=e^{-\lambda_1 t}\varphi_1$ in time $T$.

\section{Semi-global results}\label{SemiGlobal}
In this section we present two semi-global results for the exact controllability to the ground state solution of problem \eqref{u}. In the first one, Theorem \ref{teoglobal}, we show that the solution of \eqref{u} with initial condition $u_0$ lying in a suitable strip (see condition \eqref{ipu0}), reaches the ground state in finite time $T_R$.
\begin{thm}\label{teoglobal}
Let $A:D(A)\subset X\to X$ be a densely defined linear operator such that \eqref{ipA} holds, and let $B:D(B)\subset X\to X$ be an unbounded linear operator that verifies \eqref{ipB}. Let $\{A,B\}$ be a $1$-null controllable pair with control cost that satisfies \eqref{bound-control-cost}. Then there exists a constant $r_1>0$ such that for any $R>0$ there exists $T_{R}>0$ such that for all $u_0\in D(A^{1/2})$ that satisfy
\begin{equation}\label{ipu0}
\begin{array}{l}
\left|\langle u_0,\varphi_1\rangle_{1/2}-1\right|< r_1,\\\\
\norm{u_0-\langle u_0,\varphi_1\rangle_{1/2}\;\varphi_1}_{1/2}\leq R,
\end{array}
\end{equation}
problem \eqref{u} is exactly controllable to the ground state solution $\psi_1(t)=e^{-\lambda_1 t}\varphi_1$ in time $T_{R}$.
\end{thm}
Our second semi-global result, Theorem \ref{teoglobal0} below, ensures the exact controllability of all initial states $u_0\in D(A^{1/2})\setminus \varphi_1^\perp$  to the evolution of their orthogonal projection along the ground state defined by
\begin{equation}\label{exactphi1}
\phi_1(t)=\langle u_0,\varphi_1\rangle_{1/2}\; \psi_1(t), \quad\forall\, t \geq 0,
\end{equation}
where $\psi_1$ is the ground state solution.

\begin{thm}\label{teoglobal0}
Let $A:D(A)\subset X\to X$ be a densely defined linear operator such that \eqref{ipA} holds, and let $B:D(B)\subset X\to X$ be an unbounded linear operator that verifies \eqref{ipB}. Let $\{A,B\}$ be a $1$-null controllable pair with control cost that satisfies \eqref{bound-control-cost}.

Then, for any $R>0$ there exists $T_R>0$ such that for all $u_0\in D(A^{1/2})$ satisfying
\begin{equation}\label{cone}
\norm{u_0-\langle u_0,\varphi_1\rangle_{1/2}\;\varphi_1}_{1/2}\leq R \left|\langle u_0,\varphi_1\rangle_{1/2}\,\right|
\end{equation}
system \eqref{u} is exactly controllable to $\phi_1$, defined in \eqref{exactphi1}, in time $T_R$.
\end{thm}
To prove Theorems \ref{teoglobal} and \ref{teoglobal0}, one may follows the strategies described in \cite[Section 5]{acue}. The only difference with respect to \cite[Section 5]{acue} is that, in the current setting, $u_0\in D(A^{1/2})$ and this yields to the following definition of the controllability time
\begin{equation*}
T_R:=1+\frac{1}{\lambda_2}\log\left(\frac{R^2}{r^2_1}\right).
\end{equation*}

\section{Applications}\label{Applications}
In this section we discuss applications of Theorem \ref{teo-contr-B-unb} to parabolic equations. Let us first recall a result from \cite{acue} which provides sufficient conditions for a pair of linear operators $A$ and $B$, to be $j$-null controllable with control cost that fulfils \eqref{bound-control-cost}.
\begin{thm}\label{Thm-suff-cond}
Let $A:D(A)\subset X\to X$ be such that
\begin{equation}\label{ipA-sigma}
\begin{array}{ll}
(a) & A \mbox{ is self-adjoint},\\
(b) &\exists\,\sigma>0\,:\,\langle Ax,x\rangle \geq-\sigma \norm{x}^2,\,\, \forall\, x\in D(A),\\
(c) &\exists\,\lambda>\sigma\,:\,(\lambda I+A)^{-1}:X\to X \mbox{ is compact}
\end{array}
\end{equation}
and suppose that there exists a constant $\alpha>0$ for which the eigenvalues of $A$ verifies the gap condition
\begin{equation}\label{gap}
\sqrt{\lambda_{k+1}-\lambda_1}-\sqrt{\lambda_k-\lambda_1}\geq \alpha,\quad\forall\, k\in \NN^*.
\end{equation}
Let $B:D(B)\subset X\to X$ be a linear operator such that there exist $b,q>0$ for which
\begin{equation}\label{ipB-suff-cond}
\begin{array}{l}
\langle B\varphi_j,\varphi_j\rangle\neq0\quad\mbox{and}\quad\left|\lambda_k-\lambda_j\right|^q|\langle B\varphi_j,\varphi_k\rangle|\geq b,\quad\forall\,k\neq j.
\end{array}
\end{equation}
Then, the pair $\{A,B\}$ is $j$-null controllable in any time $T>0$ with control cost $N(T)$ that satisfies \eqref{bound-control-cost}.
\end{thm}
In particular, for accretive operators ($\sigma=0$), hypothesis \eqref{gap} can be replaced by
\begin{equation}\label{gap-no-l1}
\sqrt{\lambda_{k+1}}-\sqrt{\lambda_k}\geq \alpha,\quad\forall\, k\in \NN^*,
\end{equation}
see \cite[Remark 6.1]{acue}.

In order to verify the property of $1$-null controllability required in Theorem \ref{teo-contr-B-unb}, we will check the validity of the assumptions of Theorem \ref{Thm-suff-cond} for $j=1$.
\subsection{Fokker--Planck equation}\label{ex1}
The Fokker-Planck equation describes the evolution of the probability density $u(t,x)$ of a real-valued random variable $X_t$, which is associated with an Ito stochastic differential equation driven by a standard Wiener process $W_t$
\begin{equation}\label{SDE}
\begin{cases}
dX_t=\nu(t,X_t)dt+\sigma(t,X_t)dW_t\\
X(t=0)=X_0
\end{cases}
\end{equation}
with drift $\nu(t,x)$, diffusion coefficient $D(t,x)=\sigma^2(t,x)/2$ and initial condition $X_0$. We recall that given any $a\leq b$, denoting by $u(t,x)$ the probability density associated to $X_t$, the following identity holds
\begin{equation*}
\mathbb{P}(a\leq X_t\leq b )=\int_a^b u(t,x)dx.
\end{equation*}
Under suitable assumptions on the coefficients $\nu$ and $\sigma$, the equation satisfied by $u(t,x)$, named after A. D. Fokker and M. Planck, is the following one
\begin{equation*}
\begin{cases}
\frac{\partial}{\partial t}u(t,x)=\frac{\partial^2}{\partial x^2}\left(D(t,x)u(t,x)\right)-\frac{\partial}{\partial x}\left(\nu(t,x)u(t,x)\right),&(t,x)\in[0,T]\times \RR\\
u(0,x)=u_0(x)
\end{cases}
\end{equation*}
where $u_0$ is the density associated to $X_0$.

Physically, the probability density $u(t,x)$ can be interpreted as a quantity proportional to the number of particles in a flow of an abstract substance. For instance, it can reflect the concentration of this substance at the point $x$ at time $t$.

Let us first recall that \eqref{SDE} admits a strong solution (see \cite[Definition 9.1]{baldi}) which is pathwise unique (see \cite[Definition 9.4]{baldi}) under the following assumptions
\begin{equation*}\label{BaldiH1}
\text{(H1)}
\begin{cases}
\mbox{1. } (t,x) \mapsto \nu(t,x) \mbox{ and } (t,x) \mapsto \sigma(t,x) \mbox{ are measurable functions on } [0,T] \times \mathbb{R}\\

\mbox{2. }  \exists\,M>0\,:\,|\nu(t,x)|\leqslant M(1+|x|) \mbox{ and }  |\sigma(t,x)|\leqslant M(1+|x|)\,, \ \forall \ (t,x) \in [0,T] \times \mathbb{R}\\

\mbox{3. }  \exists\,L>0\,:\,|\nu(t,x)-\nu(t,y)|\leqslant L|x-y| \mbox{ and } |\sigma(t,x)-\sigma(t,y)|\leqslant L|x-y|\,, \ \forall \ (t,x) \in [0,T] \times \mathbb{R}
\end{cases}
\end{equation*}
(see \cite[Theorem 9.2]{baldi}).

We are interested in studying the possibility to find a drift $\nu$ such that the probability density $u$ reaches the associated ground state solution in finite time. Moreover, as suggested by the results of the previous sections, we would like to take a drift of the form
\begin{equation*}
\nu(t,x)=p(t)\mu(x),\qquad (t,x)\in[0,T]\times \RR.
\end{equation*}
However, the sublinear growth assumption $2.$ of (H1) on $\nu$ requires an essential bound on the scalar control $p$. Since our controls $p$ are only locally square integrable, we need to weaken (H1). Thus, in Appendix \ref{AppendixB}, we adapt the proof of the existence and uniqueness of a strong solution of \eqref{SDE} under the following weaker assumptions:
\begin{equation*}\label{ABCUH2}
\text{(H2)}
\begin{cases}
\mbox{1. } (t,x) \mapsto \sigma(t,x) \mbox{ is a measurable function on } [0,T] \times \RR\\

\mbox{2. }  p \in L^2_{loc}(\mathbb{R})\\

\mbox{3. }  \exists\,M>0\,:\,|\mu(x)|\leqslant M(1+|x|) \mbox{ and}  |\sigma(t,x)|\leqslant M(1+|x|)\,, \ \forall \ (t,x) \in [0,T] \times \RR\\

\mbox{4. }  \exists\,L>0\,:\,|\mu(x)-\mu(y)|\leqslant L|x-y| \,, \; |\sigma(t,x)-\sigma(t,y)|\leqslant L|x-y|\,, \; \forall \ t \in [0,T]\,, \;  \forall \ x,y  \in \RR
\end{cases}
\end{equation*}

From now on we will consider a constant diffusion $\sigma(t,x)\equiv\sqrt{2}$ and drift of the form $\nu(t,x)=p(t)\mu(x)$, where $p\in L^2(0,T)$ and $\mu:[0,1]\to \RR$ is at least Lipschitz continuous. Then, by extending $\mu(\cdot)$ outside the interval $[0,1]$ as follows
\begin{equation}\label{extension-mu}
\tilde{\mu}(x)=\begin{cases}
\mu(0)&x\leq0\\
\mu(x)&0<x<1\\
\mu(1)& x\geq1
\end{cases}
\end{equation}
it is clear that $\tilde{\mu}$ satisfies assumption (H2).

The Fokker-Planck equation can be studied also on bounded domains under suitable boundary conditions such as perfectly reflecting, partially reflective, and non reflecting boundary conditions. We now describe such conditions for the equation
\begin{equation*}
\begin{cases}
\frac{\partial}{\partial t}u(t,x)=\frac{\partial^2}{\partial x^2}u(t,x)-p(t)\frac{\partial}{\partial x}\left(\mu(x)u(t,x)\right),&(t,x)\in[0,T]\times [0,1]\\
u(0,x)=u_0(x)
\end{cases}
\end{equation*}
with their physical aspects and the mathematical difficulties they generate.

\vskip 4mm

$\bullet$ The perfectly reflecting boundary conditions are given by:
\begin{equation}
\frac{\partial}{\partial x}u(t,1)-p(t)\mu(1)u(t,1)=\frac{\partial}{\partial x}u(t,0)-p(t)\mu(0)u(t,0)=0
\end{equation}
Let us note that, for such boundary conditions, the total mass is conserved in the interval $[0,1]$. Thus, if the initial data $u_0$ has a total mass $1$, then through time, the probability density $u$ satisfies
\begin{equation}\label{prob=1}
\int_0^1 u(t,x)dx=1\qquad\forall\,t\in (0,T),
\end{equation}
that is, the probability to find particles in the interval $[0,1]$ is equal to 1. Indeed, by using the equation solved by $u$ we get
\begin{equation*}
\int_0^1 \frac{\partial}{\partial t}u(t,x)dx=\int_0^1\left(\frac{\partial^2}{\partial x^2}u(t,x)-p(t)\frac{\partial}{\partial x}\left(\mu(x)u(t,x)\right)\right)dx=0,
\end{equation*}
hence 
\begin{equation*}
\frac{\partial}{\partial t}\int_0^1 u(t,x)dx=0\qquad\forall\,t\in(0,T)
\end{equation*}
which implies condition \eqref{prob=1}. However, these perfect boundary conditions have a great impact on the functional frame in which one can set up the abstract formulation of the Fokker-Planck equation: the domain of the abstract operator $A$ will have to include the dependence on time and control $p$, that is, one should have to handle $D(A(t,p))=\{u(t,\cdot) \in H^2(0,1), u_x(t,1)=p(t)\mu(1)u(t,1)\,, u_x(t,0)=p(t)\mu(0)u(t,0)\}$ for all $t \in [0,T]$. As far as we know such mathematical difficulties have not yet been studied in bilinear control, the strongest difficulty being that the domain depends itself on the scalar control $p$.

\vskip 4mm

$\bullet$ The partially reflecting boundary conditions (i.e. reflecting only the diffusive part of the process, the Brownian motion), are given by
$$
\frac{\partial}{\partial x}u(t,1)=\frac{\partial}{\partial x}u(t,0)=0.
$$
which leads to the following problem
\begin{equation}\label{FP-Neumann}
\left\{\begin{array}{ll}
\frac{\partial}{\partial t}u(t,x)-\frac{\partial^2}{\partial x^2} u(t,x)+p(t)\frac{\partial}{\partial x}\left(\mu(x)u(t,x)\right)=0&(t,x)\in[0,T]\times[0,1]\\\\
\frac{\partial}{\partial x}u(t,1)=\frac{\partial}{\partial x}u(t,0)=0&t\in[0,T]\\\\
u(0,x)=u_0(x)&x\in[0,1]
\end{array}\right.
\end{equation}
However, when considering such boundary conditions, the total mass is, in general, no more conserved through time. Indeed, considering again the equation solved by $u$, this time we find that
\begin{equation*}
\int_0^1 \frac{\partial}{\partial t}u(t,x)dx=\int_0^1\left(\frac{\partial^2}{\partial x^2}u(t,x)-p(t)\frac{\partial}{\partial x}\left(\mu(x)u(t,x)\right)\right)dx=p(t)[\mu(1)u(t,1)-\mu(0)u(t,0)].
\end{equation*}

\vskip 4mm

$\bullet$ The absorbing (Dirichlet) boundary conditions
$$
u(t,1)=u(t,0)=0
$$
leads to following problem
\begin{equation}\label{FP-Dirichlet}
\left\{\begin{array}{ll}
\frac{\partial}{\partial t}u(t,x)-\frac{\partial^2}{\partial x^2} u(t,x)+p(t)\frac{\partial}{\partial x}\left(\mu(x)u(t,x)\right)=0&(t,x)\in[0,T]\times[0,1]\\\\
u(t,1)=u(t,0)=0&t\in[0,T]\\\\
u(0,x)=u_0(x)&x\in[0,1].
\end{array}\right.
\end{equation}
The total density is again in general, not preserved 
\begin{equation*}
\int_0^1 \frac{\partial}{\partial t}u(t,x)dx=\int_0^1\left(\frac{\partial^2}{\partial x^2}u(t,x)-p(t)\frac{\partial}{\partial x}\left(\mu(x)u(t,x)\right)\right)dx=\frac{\partial}{\partial x}u(t,1)-\frac{\partial}{\partial x}u(t,0).
\end{equation*}
In view of the mathematical difficulties generated by the perfectly reflecting boundary conditions, we shall consider in the present paper, only the partially reflecting boundary conditions, and the absorbing ones.

\vskip 2mm

Let us start by studying local controllability to the ground state for problem \eqref{FP-Neumann}, where our bilinear control will be the time dependent part of the drift $p(\cdot)$. We set $I=(0,1)$. We recast the problem in the general setting of \eqref{u} by introducing the operators $A$ and $B$ defined by
\begin{equation*}
D(A)=\left\{\varphi\in H^2(I)\,:\, \frac{\partial}{\partial x}\varphi(0)=\frac{\partial}{\partial x}\varphi(1)=0\right\},\quad A\varphi=-\frac{d^2\varphi}{dx^2}
\end{equation*}
\begin{equation*}
D(B)=\left\{\varphi\in L^2(I)\,:\,\frac{d}{dx}(\mu\varphi)\in L^2(I)\right\},\quad B\varphi=\frac{d}{dx}\left(\mu\varphi\right)
\end{equation*}
where $\mu$ is a real-valued function in $H^2(I)$ to be chosen later on in order to fulfill the rank condition \eqref{ipB-suff-cond}.

$A$ satisfies all the properties in \eqref{ipA} and its eigenvalues and eigenfunctions have the following explicit expressions
\begin{equation*}
\begin{array}{lll}
\lambda_0=0,&\varphi_0=1\\\\
\lambda_k=(k\pi)^2,& \varphi_k(x)=\sqrt{2}\cos(k\pi x),& \forall\, k\geq1.
\end{array}
\end{equation*}
It is straightforward to prove that the eigenvalues fulfill the required gap property. Indeed, 
\begin{equation*}
\sqrt{\lambda_{k+1}}-\sqrt{\lambda_k}=(k+1)\pi-k\pi=\pi,\qquad \forall k\in \NN \,,
\end{equation*}
so that \eqref{gap-no-l1} is satisfied. 

Observe that, in this context, we have an explicit description of the spaces $D(A^{s/2})$, see \cite[Section 4.3.3]{tri} for a general result. For example, for $s=1$
\begin{equation*}
D(A^{1/2})=H^1(I).
\end{equation*}
In order to apply Theorem \ref{teo-contr-B-unb}, we have to check that $D(A^{1/2})\hookrightarrow D(B)$. This is easily proved as follows:
\begin{equation*}
||\varphi||_{D(B)}=||(\mu\varphi)_x||_{L^2(I)}\leq C_\mu(||\varphi_x||_{L^2(I)}+||\varphi||_{L^2(I)})\leq C||\varphi||_{D(A^{1/2})}
\end{equation*}
for all $\varphi$ in $D(A^{1/2})$.

\vskip 2mm

To prove local controllability of \eqref{FP-Neumann} to the ground state solution $\psi_0(t,x)\equiv 1$, we want to use Theorem \ref{Thm-suff-cond} for $j=0$ (note also that, due to the Neumann boundary conditions, one has to consider $j=0$ instead of $j=1$ and for $k$ varying in $\mathbb{N}$ instead of $\mathbb{N}^{\ast}$) in order to apply Theorem \ref{teo-contr-B-unb}. Thus, we have to check that \eqref{ipB-suff-cond} is satisfied for all $k \in \mathbb{N}$. 
By definition of $B$, we have $B\varphi_0=\mu'$, hence one can observe that the choice $\mu(x)=x$ for all $x \in I$ leads to $\langle B\varphi_0,\varphi_k\rangle=0$, which is not suitable for our purposes.

Let us first examine more precisely $\langle B\varphi_0,\varphi_k\rangle$ for all $k \in \mathbb{R}^{\ast}$:
\begin{equation*}
\begin{split}
\langle B\varphi_0,\varphi_k\rangle&=\sqrt{2}\int_0^1\mu'(x)\cos(k\pi x)dx=\sqrt{2}\left.\mu'(x)\frac{\sin (k\pi x)}{k\pi}\right|^1_0-\sqrt{2}\int_0^1\mu''(x)\frac{\sin (k\pi x)}{k\pi}dx\\
&=\left.\sqrt{2}\mu''(x)\frac{\cos(k\pi x)}{(k\pi)^2}\right|^1_0-\sqrt{2}\int_0^1\mu'''(x)\frac{\cos(k\pi x)}{(k\pi)^2}dx\\
&=\frac{\sqrt{2}}{(k\pi)^2}\left[\mu''(1)(-1)^k-\mu''(0)\right]-\sqrt{2}\int_0^1\mu'''(x)\frac{\cos(k\pi x)}{(k\pi)^2}dx,\qquad\forall\,k\geq1
\end{split}
\end{equation*}
By the Riemann-Lebesgue Lemma, the last integral term on the right-hand side of the above identity converges to $0$ as $k$ goes to $+\infty$. Thus, if we choose $\mu$ such that $\mu''(1)\neq\pm \mu''(0)$, then there exists $k_1 \geq 1$, such that
\begin{equation*}
\exists\,b>0\,:\,\lambda_k|\langle B\varphi_0,\varphi_k\rangle|\geq b\,, \quad \mbox{ for all } k > k_1
\end{equation*}
Furthermore, if $\mu(1)\neq \mu(0)$ we deduce that
\begin{equation*}
\langle B\varphi_0,\varphi_0\rangle=\int_0^1\mu'(x)dx\neq0.
\end{equation*}
Hence, if $\mu''(1)\neq\pm \mu''(0)$ and $\mu(1)\neq \mu(0)$, we have only to check that $\langle B\varphi_0,\varphi_k\rangle  \neq 0$ for a finite range of $k \in \{1, \ldots, k_1\}$.
To sum up, every function $\mu$ such that
\begin{equation*}
\left\{\begin{array}{l}
\mu(1)\neq \mu(0)\\\\
\mu''(1)\neq \pm\mu''(0)\\\\
\langle \mu\varphi_0,\varphi_k\rangle\neq 0,\qquad k=1,\dots,k_1
\end{array}
\right.
\end{equation*}
is suitable for our controllability purposes. For instance, any power $\mu(x)=x^n$, $x\in [0,1]$, with $n>2$ satisfies the above conditions at the boundary. Moreover, once extending $\mu$ as in \eqref{extension-mu}, it clearly satisfies (H2)\footnote{We observe that by multiplying $\mu$ by a cut-off function, it is possible to construct a smooth extension of $\mu$ which satisfies (H2).}. So, one has just to check that $\langle \mu\varphi_0,\varphi_k\rangle\neq 0$ for $k=1,\dots,k_1$. For example, for $n=3$ one easily prove that
\begin{equation*}
\langle \mu\varphi_0,\varphi_k\rangle=\begin{cases}
6\sqrt{2}\frac{(-1)^k}{(k\pi)^2}& k\geq1\\\\
\sqrt{2}&k=0.
\end{cases}
\end{equation*}
Another suitable choice is $\mu(x)=\sin(\alpha x)$ for all $x \in I$, where $\alpha$ is as any positive real number chosen in $[0,\infty) \backslash \pi \mathbb{N}$. In this case one can check that $ \mu''(1)=-\alpha^2\sin(\alpha) \neq 0 =\mu''(0)$, $\mu(1)= \sin(\alpha)\neq 0= \mu(0)$, and
$$
\langle B\varphi_0,\varphi_k\rangle=\frac{\sqrt{2}}{\alpha^2-(k\pi)^2}(-1)^k\alpha^2 \sin(\alpha).
$$
Thus, any choice for $\mu$ of the above forms meets all the required properties, and also the assumptions in (H2) for the well-posedness of the stochastic differential equation.

Hence, thanks to Theorems \ref{Thm-suff-cond} and \ref{teo-contr-B-unb}, is possible to build a control $p\in L^2(0,T)$ such that the solution of the Fokker--Planck equation \eqref{FP-Neumann}, with initial condition in a neighbourhood of $\varphi_0=1$, partially reflecting boundary conditions and drift $\nu(t,x)=p(t)\mu(x)$, is controllable to $\psi_0=1$ in time $T$. This means that the probability to find the particle in the interval $[0,1]$ at time $T$ is equal to 1 (then the event happens almost surely) even if we are in presence of non-perfectly reflecting walls. This is due to the appropriate choice of the drift.

We move now to the Fokker-Planck equation with absorbing (Dirichlet) boundary conditions. In this case the eigenvalues and eigenfunctions of the Laplacian are the following
\begin{equation*}
\lambda_k=(k\pi)^2,\quad \varphi_k(x)=\sqrt{2}\sin(k\pi x),\quad \forall k\in\NN^*.
\end{equation*}
Since we have that
\begin{equation*}
\int_0^1\varphi_1(x)dx=\sqrt{2}\int_0^1 \sin(\pi x)dx=\frac{2\sqrt{2}}{\pi}
\end{equation*}
controlling the solution to the ground state means that we are forcing some mass to remain in the interval $[0,1]$ after time $T$ (in the sense that with probability equal to $\frac{2\sqrt{2}}{\pi}\cong 0.9$ we find a particle in the interval $[0,1]$), even though we are in presence of absorbing boundary conditions. 

In order to apply Theorem \ref{Thm-suff-cond} for $j=1$, and deduce the controllability of \eqref{FP-Dirichlet} to the ground state by applying Theorem \ref{teo-contr-B-unb}, we have to verify the rank condition of the unbounded operator $B$. Let us compute the following scalar product
\begin{equation*}
\begin{split}
\langle (\mu\varphi_1)',\varphi_k\rangle&=\sqrt{2}\int_0^1 (\mu\varphi_1)'(x)\sin(k\pi x)dx\\
&=\sqrt{2}\left(-\left.(\mu\varphi_1)'(x)\frac{\cos(k\pi x)}{k\pi}\right|^1_0+\int_0^1\left(\mu\varphi_1\right)''(x)\frac{\cos(k\pi x)}{k\pi}dx\right)\\
&=\frac{2}{k}\left(\mu(1)(-1)^{k}+\mu(0)\right)+\frac{\sqrt{2}}{k\pi}\int_0^1\left(\mu\varphi_1\right)''(x)\cos(k\pi x)dx.
\end{split}
\end{equation*}
Thus, if $\mu(1)\neq\pm \mu(0)$ and $\langle (\mu\varphi_1)',\varphi_k\rangle \neq0,\,\forall\,k\in\NN^*$, there exists a constant $b$ such that
\begin{equation*}
\lambda_k^{1/2}|\langle B\varphi_1,\varphi_k\rangle|\geq b,\qquad\forall\,k\in\NN^*.
\end{equation*}
Hence, any potential $\mu(x)=x^n$, with $n\geq 1$, extended to the real line as in the previous example, is admissible for having exact controllability of \eqref{FP-Dirichlet} to the ground state. 
\begin{oss}
\emph{If we consider $\mu(x)=x$, we can directly check that the Fourier coefficients of $B\varphi_1$ do not vanish for every $k\in\NN$}
\begin{equation*}
\langle (\mu\varphi_1)',\varphi_k\rangle=\left\{\begin{array}{ll}
\frac{(-1)^k2k}{k^2-1},&k\geq2\\\\
\frac{1}{2}&k=1.
\end{array}\right.
\end{equation*}
\emph{and the lower bound \eqref{ipB-suff-cond} is satisfied with $q=\frac{1}{2}$ and $b=2\pi$}.
\end{oss}


\subsection{Diffusion equation with Neumann boundary conditions}\label{ex4}
Now, we consider a diffusion equation with a controlled potential subject to Neumann boundary conditions. Let $I=(0,1)$ and consider the following problem
\begin{equation}\label{40}
\left\{\begin{array}{ll}
u_t(t,x)-\partial^2_{x}u(t,x)+p(t)\mu(x)\left(u_x(t,x)+u(t,x)\right)=0 & (t,x)\in[0,T]\times I \\
u_x(t,0)=0,\quad u_x(t,1)=0 &t\in[0,T]\\
u(0,x)=u_0(x). & x\in I
\end{array}\right.
\end{equation}

Let $X=L^2(I)$, we rewrite \eqref{40} in abstract form by defining the operators $A$ and $B$ as
\begin{equation}\nonumber
D(A)=\{ \varphi\in H^2(I): \varphi^\prime(0)=0\,,\varphi^\prime(1)=0\},\quad A\varphi=-\frac{d^2\varphi}{dx^2}
\end{equation}
\begin{equation}\nonumber
D(B)=H^1(I),\quad B\varphi=\mu\left(\frac{d}{dx}\varphi+\varphi\right).
\end{equation}
where $\mu$ is a real-valued function in $H^2(I)$.

Operator $A$ satisfies the assumptions in \eqref{ipA} and it is possible to compute explicitly its eigenvalues and eigenfunctions:
\begin{equation*}
\begin{array}{lll}
\lambda_0=0,&\varphi_0=1\\
\lambda_k=(k\pi)^2,& \varphi_k(x)=\sqrt{2}\cos(k\pi x),& \forall\, k\geq1.
\end{array}
\end{equation*}

Since the eigenvalues are the same of those in Example \ref{ex1} for $k\geq1$, the gap condition is satisfied for all $k\geq0$.

Furthermore, we have that
\begin{equation*}
\norm{B\varphi}=\norm{\mu(\varphi_x+\varphi)}\leq C_\mu(\norm{A^{1/2}\varphi}+\norm{\varphi})\leq C\norm{\varphi}_{D(A^{1/2})},
\end{equation*}
thus, also hypothesis \eqref{BA12} is verified.

Let us compute the scalar product $\langle B\varphi_0,\varphi_k\rangle$ to find a lower bound of the Fourier coefficients of $B\varphi_0$:
\begin{equation*}\begin{split}
\langle \mu\left(\varphi_0^\prime+\varphi_0\right),\varphi_k\rangle&=\sqrt{2}\int_0^1 \mu(x)\cos(k\pi x)dx\\
&=\sqrt{2}\left(\left.\mu(x)\frac{\sin(k\pi x)}{k\pi}\right|^1_0-\int_0^1\mu'(x)\frac{\sin(k\pi x)}{k\pi}dx\right)\\
&=\sqrt{2}\left(\left.\mu'(x)\frac{\cos(k\pi x)}{(k\pi)^2}\right|^1_0-\int^1_0\mu''(x)\frac{\cos(k\pi x)}{(k\pi)^2}dx\right)\\
&=\frac{\sqrt{2}}{(k\pi)^2}\left(\mu'(1)(-1)^k-\mu'(0)\right)-\frac{\sqrt{2}}{(k\pi)^2}\int^1_0\mu''(x)\cos(k\pi x)dx.
\end{split}
\end{equation*}
Thus, reasoning as Example \ref{ex1}, if $\langle B\varphi_0,\varphi_k\rangle\neq0$ $\forall k \in\NN$ and $\mu\rq{}(1)\pm\mu\rq{}(0)\neq 0$, then we have that 
\begin{equation}\label{lb3}
\exists\,\, C>0\mbox{ such that } |\langle B\varphi_0,\varphi_k\rangle|\geq Ck^{-2}=C\lambda_k^{-1},\quad \forall k\in\NN^*
\end{equation}
and therefore hypothesis \eqref{ipB-suff-cond} is fulfilled. 
\begin{oss}
\emph{An example of a suitable function $\mu$ for problem \eqref{40} that satisfies the above hypothesis, is $\mu(x)=x^2$, for which}
\begin{equation*}
\langle B\varphi_0,\varphi_k\rangle=\frac{2\sqrt{2}(-1)^k}{(k\pi)^2},\quad \forall \ k \geq 1,\,\, \mbox{ \emph{and} }\quad \langle B\varphi_0,\varphi_0\rangle=1/3.
\end{equation*}
\end{oss}
Applying Theorem \ref{teo-contr-B-unb}, it follows that system \eqref{40} is exactly controllable to $\psi_0$.

\subsection{Degenerate parabolic equation with Dirichlet boundary conditions}
Let $T>0$, $I=(0,1)$, $X=L^2(I)$ and consider the following degenerate bilinear control system
\begin{equation}\label{eq-ex-deg-D}
\left\{
\begin{array}{ll}
u_t-\left(x^{\alpha} u_x\right)_x+p(t)\mu(x)u_x=0,& (t,x)\in [0,T]\times I
\vspace{.1cm}\\
u(t,1)=0,\quad\left\{\begin{array}{ll} u(t,0)=0,& \mbox{ if }\alpha\in[0,1)\vspace{.1cm},\\
\left(x^{\alpha}u_x\right)(t,0)=0,& \mbox{ if }\alpha\in[1,2),\end{array}\right.
\vspace{.1cm}\\
u(0,x)=u_0(x),&t\in x\in I.
\end{array}
\right.
\end{equation}
with $\mu(x)=x$ and $\alpha\in[0,2)$ the degeneracy parameter. We recall that when $\alpha\in[0,1)$ problem \eqref{eq-ex-deg-D} is called \emph{weakly degenerate}, while for $\alpha\in[1,2)$ \emph{strongly degenerate}.

Let $\alpha\in[0,1)$ and we define the weighted Sobolev spaces
\begin{equation}\label{sob-sp-w}
\begin{array}{l}
H^1_{\alpha}(I)=\left\{u\in X: u \mbox{ is absolutely continuous on } [0,1], x^{\alpha/2}u_x\in X\right\}\vspace{.1cm}\\
H^1_{\alpha,0}(I)=\left\{u\in H^1_\alpha(I):\,\, u(0)=0,\,\,u(1)=0\right\}\vspace{.1cm}\\
H^2_\alpha(I)=\left\{u\in H^1_\alpha(I): x^{\alpha}u_x\in H^1(I)\right\},
\end{array}
\end{equation}
and the linear operator $A:D(A)\subset X\to X$ by
\begin{equation*}
\left\{\begin{array}{l}
\forall u\in D(A),\quad Au:=-(x^{\alpha}u_x)_x,\vspace{.1cm}\\
D(A):=\{u\in H^1_{\alpha,0}(I),\,\, x^{\alpha}u_x\in H^1(I)\}.
\end{array}\right.
\end{equation*}
For $\alpha\in[1,2)$, we define the spaces
\begin{equation}\label{sob-sp-s}
\begin{array}{l}
H^1_{\alpha}(I)=\left\{u\in X: u \mbox{ is absolutely continuous on } (0,1],\,\, x^{\alpha/2}u_x\in X\right\}\vspace{.1cm}\\
H^1_{\alpha,0}(I):=\left\{u\in H^1_{\alpha}(I):\,\,u(1)=0\right\},\vspace{.1cm}\\
H^2_\alpha(I)=\left\{u\in H^1_\alpha(I):\,\, x^{\alpha}u_x\in H^1(I)\right\}
\end{array}
\end{equation}
and the linear degenerate operator $A:D(A)\subset X\to X$ by
\begin{equation*}
\left\{\begin{array}{l}
\forall u\in D(A),\quad Au:=-(x^{\alpha}u_x)_x,\vspace{.1cm}\\
D(A):=\left\{u\in H^1_{\alpha,0}(I):\,\, x^{\alpha}u_x\in H^1(I)\right\}\vspace{.1cm}\\
\qquad\,\,\,\,\,=\left\{u\in X:\,\,u \mbox{ is absolutely continuous in (0,1] },\,\, x^{\alpha}u\in H^1_0(I),\right.\vspace{.1cm}\\
\qquad\qquad\,\,\,\left.x^{\alpha}u_x\in H^1(I)\mbox{ and } (x^{\alpha}u_x)(0)=0\right\}.
\end{array}\right.
\end{equation*}
In both cases of weak and strong degeneracy it can be proved that $D(A)$ is dense in X and $A:D(A)\subset X\to X$ is a self-adjoint accretive operator (see \cite{cmp} and \cite{cmvn}, respectively). Therefore, $-A$ is the infinitesimal generator of an analytic semigroup of contractions on $X$. Furthermore, it has been showed (see, for instance, \cite[Appendix]{acf}) that $A$ has a compact resolvent. 

Let $\alpha\in [0,2)$ and define 
\begin{equation*}
\nu_\alpha:=\frac{|1-\alpha|}{2-\alpha},\qquad k_\alpha:=\frac{2-\alpha}{2}.
\end{equation*}
Given $\nu\geq0$, we denote by $J_\nu$ the Bessel function of the first kind and order $\nu$ and by $j_{\nu,1}<j_{\nu,2}<\dots<j_{\nu,k}<\dots$ the sequence of all positive zeros of $J_\nu$. It is possible to prove that the eigenvalues and eigenfunctions of $A$ are given by
\begin{equation}\label{13}
\lambda_{\alpha,k}=k^2_{\alpha}j^2_{\nu_\alpha,k},
\end{equation}
\begin{equation}\label{14}
\varphi_{\alpha,k}(x)=\frac{\sqrt{2k_\alpha}}{|J'_{\nu_\alpha}(j_{\nu_\alpha,k})|}x^{(1-\alpha)/2}J_{\nu_\alpha}\left(j_{\nu_\alpha,k}x^{k_\alpha}\right)
\end{equation}
for every $k\in\NN^*$. Moreover, the family $\left(\varphi_{\alpha,k}\right)_{k\in\NN^*}$ is an orthonormal basis of $X$, see \cite{g}.
Hereafter, we shall denote the eigenfunctions $\left(\varphi_{\alpha,k}\right)_{k\in\NN^*}$ by $\left(\varphi_{k}\right)_{k\in\NN^*}$, and the eigenvalues $\left(\lambda_{\alpha,k}\right)_{k\in\NN^*}$ by $\left(\lambda_{k}\right)_{k\in\NN^*}$.

It has been proved (see \cite[page 135, Proposition 7.8]{kl} and \cite[Corollary 1]{cmv2}) that the gap condition is satisfied for all $\alpha\in [0,2)$. More precisely, it is proved that:
\begin{itemize}
\item if $\alpha\in [0,1)$, $\nu_{\alpha}=\frac{1-\alpha}{2-\alpha}\in\left(0,\frac{1}{2}\right]$, the sequence $\left( j_{\nu_\alpha,k+1}-j_{\nu_\alpha,k}\right)_{k\in\NN^*}$ is nondecreasing and converges to $\pi$. Therefore, 
\begin{equation*}
\sqrt{\lambda_{k+1}}-\sqrt{\lambda_k}=k_{\alpha}\left( j_{\nu_\alpha,k+1}-j_{\nu_\alpha,k}\right)\geq k_{\alpha}\left( j_{\nu_\alpha,2}-j_{\nu_\alpha,1}\right)\geq  \frac{7}{16}\pi>0,
\end{equation*}
\item if $\nu_{\alpha}\geq \frac{1}{2}$, the sequence $\left( j_{\nu_\alpha,k+1}-j_{\nu_\alpha,k}\right)_{k\in\NN^*}$ is nonincreasing and converges to $\pi$. Thus,
\begin{equation*}
\sqrt{\lambda_{k+1}}-\sqrt{\lambda_k}=k_{\alpha}\left( j_{\nu_\alpha,k+1}-j_{\nu_\alpha,k}\right)\geq k_{\alpha}\pi\geq \frac{(2 - \alpha)\pi}{2}> 0.
\end{equation*}
\end{itemize}
Therefore, the hypothesis \eqref{gap} is fulfilled in both weak and strong degenerate problems with different constants.
We now define the unbounded linear operator $B$ by
\begin{equation*}
\begin{split}
B: D(B)=H^1(I)\subset X&\to X\\
\varphi&\mapsto \mu(x)\varphi^\prime
\end{split}
\end{equation*}
where we recall that $\mu(x)=x$. We observe that $D(A^{1/2})\subset D(B)$ and
\begin{equation*}
||\varphi||_{D(B)}=||C\mu\varphi'||\leq C||x^{\alpha/2}\varphi'||\leq C ||\varphi||_{D(A^{1/2})},\quad \forall\,\varphi\in D(A^{1/2}).
\end{equation*}
Hence \eqref{BA12} holds true. Finally, we need to prove the validity of \eqref{ipB-suff-cond}. To this purpose, we study the Fourier coefficients of $B\varphi_1$:
\begin{equation*}
\begin{split}
\langle B\varphi_1,\varphi_k\rangle&=\int_0^1x\varphi_1^\prime(x)\varphi_k(x)dx=-\frac{1}{\lambda_k}\int_0^1x\varphi_1^\prime(x)\left(x^\alpha \varphi^\prime_k(x)\right)^\prime dx\\
&=-\frac{1}{\lambda_k}\left[\left.x\varphi_1^\prime(x) x^\alpha \varphi_k^\prime(x)\right|^1_0-\int_0^1(x\varphi_1^\prime)^\prime(x) x^\alpha \varphi_k^\prime(x)dx\right]\\
&=-\frac{1}{\lambda_k}\left[\varphi_1^\prime(1)\varphi_k^\prime(1)-\int_0^1(x\varphi_1^\prime)^\prime(x) x^\alpha \varphi_k^\prime(x)dx\right]
\end{split}
\end{equation*}
where we used \cite[Proposition 2.5, Property (2.14)]{FABCL}) to prove that
$$
\lim_{x \leftarrow 0^+}\big(x\varphi_1^\prime(x) x^\alpha \varphi_k^\prime(x)\big)=0, 
$$
so that, we have
\begin{equation*}
\begin{split}
\langle B\varphi_1,\varphi_k\rangle&=-\frac{1}{\lambda_k}\biggl[\varphi_1^\prime(1)\varphi_k^\prime(1)-\left.\left(x\varphi_1^{\prime}\right)^\prime(x)x^\alpha\varphi_k(x)\right|^1_0\\
&\quad\left.+\int_0^1\Big((x\varphi_1^\prime)^\prime x^\alpha\Big)^{\prime}(x)\varphi_k(x)dx\right]
\end{split}
\end{equation*}
We claim that for all $\alpha \in [0,2)$, the following property holds:
$$
\left.\left(x\varphi_1^{\prime}(x)\right)^{\prime}x^\alpha\varphi_k(x)\right|_{x=0} \mbox{ is defined and is equal to } 0
$$ 
Let  us consider $x \in (0,1)$. We have
\begin{equation}\label{Rvarphi1}
(x\varphi_1^{\prime})^{\prime}(x)=x^{1-\alpha} (x^{\alpha}\varphi_1^{\prime})^{\prime}(x) + (1-\alpha)\varphi_1^{\prime}(x)=-
\lambda_1 x^{1-\alpha}\varphi_1 + (1-\alpha)\varphi_1^{\prime}(x).
\end{equation}
Thus, we obtain
$$
x^{\alpha}(x\varphi_1^{\prime})^{\prime}(x)\varphi_k(x)=-\lambda_1x \varphi_1(x)\varphi_k(x)+ (1-\alpha)(x^{\alpha}\varphi_1^{\prime}(x))\varphi_k(x)
$$
Observe that
$$
\big| x \varphi_1(x)\varphi_k(x) \big| \leqslant \dfrac{1}{2}\big(x\varphi_1^2(x) + x\varphi_k^2(x)\big),\quad \ \forall \ x \in (0,1].
$$
Since $\varphi_1$ and  $\varphi_k$ are in $H^1_{\alpha}(I)$, we can use \cite[Proposition 2.5, Property (2.12)]{FABCL}) in the particular case $a(x)=x^{\alpha}$ for $x \in [0,1]$, and successively with $u=\varphi_1$, and $u=\varphi_k$. Hence, we have 
$$
\lim_{x \leftarrow 0^+}\big(x\varphi_1^2(x) + x\varphi_k^2(x)\big)=0,
$$
so that 
$$
\lim_{x \leftarrow 0^+}\big(x \varphi_1(x)\varphi_k(x)\big)=0.
$$
Since $\varphi_1 \in D(A)$ and  $\varphi_k$ is in $H^1_{\alpha}(I)$, we can use \cite[Proposition 2.5, Property (2.15)]{FABCL}), in the particular case $a(x)=x^{\alpha}$ for $x \in [0,1]$, $\phi= \varphi_k$, and $u=\varphi_1$, noticing in addition that we have $\phi(0)=0$ if
$\alpha \in [0,1)$. Hence, we deduce
$$
\lim_{x \leftarrow 0^+}\big(x^{\alpha}\varphi_1^{\prime}(x))\varphi_k(x)\big)=0.
$$
Therefore, we have proved that for all $\alpha \in [0,2)$
$$
\left.\left(x\varphi_1^{\prime}(x)\right)^{\prime}x^\alpha\varphi_k(x)\right|_{x=0} \mbox{ is defined and is equal to } 0.
$$
We easily prove that since $\varphi_1(1)=\varphi_k(1)=0$, we have
$$
\left.\left(x\varphi_1^{\prime}(x)\right)^{\prime}x^\alpha\varphi_k(x)\right|_{x=1}= 0.
$$
Thus, we obtain that
$$
\left.\left(x\varphi_1^{\prime}\right)^\prime(x)x^\alpha\varphi_k(x)\right|^1_0=0,
$$
so that using this property in our previous computations for $\langle B\varphi_1,\varphi_k\rangle$, we get
\begin{equation*}
\begin{split}
\langle B\varphi_1,\varphi_k\rangle&=-\frac{1}{\lambda_k}\left[\varphi_1^\prime(1)\varphi_k^\prime(1)+\int_0^1\Big((x\varphi_1^\prime)^\prime x^\alpha\Big)^{\prime}(x)\varphi_k(x)dx\right].
\end{split}
\end{equation*}
Now we use the identity \eqref{Rvarphi1} which yields 
\begin{equation*}
\begin{split}
\Big(x^{\alpha}(x\varphi_1^\prime)^\prime\Big)^{\prime}(x)\varphi_k(x)&=\Big(-\lambda_1x\varphi_1(x) +(1-\alpha)x^{\alpha}\varphi_1^{\prime}(x)
\Big)^{\prime}(x)\varphi_k(x)\\
&=\Big[-(2-\alpha)\lambda_1\varphi_1(x)\varphi_k(x) - \lambda_1x\varphi_1^{\prime}(x)\varphi_k(x)\Big].
\end{split}
\end{equation*}
Using this identity in our previous computations and the property that the family $\left(\varphi_{k}\right)_{k\in\NN^*}$ is an orthonormal basis of $X$, we obtain
\begin{equation*}
\langle B\varphi_1,\varphi_k\rangle=-\frac{\varphi_1^\prime(1)\varphi_k^\prime(1)}{\lambda_k-\lambda_1}.
\end{equation*}
We have proved in \cite[page 12, formula (2.43)]{cu} that
\begin{equation*}
\varphi_1^\prime(1)\varphi_k^\prime(1)=\frac{2k_{\alpha}^3j_{\nu_\alpha,1}j_{\nu_\alpha,k}}{|J'_{\nu_\alpha}(j_{\nu_\alpha,1})||J'_{\nu_\alpha}(j_{\nu_\alpha,k})|}J'_{\nu_\alpha}(j_{\nu_\alpha,1})J'_{\nu_\alpha}(j_{\nu_\alpha,k}).
\end{equation*}
The above identity, together with \eqref{13}, imply that there exists a constant $C>0$ such that
\begin{equation*}
|\langle B\varphi_1,\varphi_k\rangle|\geq C\lambda_k^{-1/2},\quad\forall\,k>1.
\end{equation*}
For $k=1$, $\langle B\varphi_1,\varphi_1\rangle$ is given by
\begin{equation*}
\langle B\varphi_1,\varphi_1\rangle=\int_0^1x\varphi_1^\prime(x)\varphi_1(x)dx=\frac{1}{2}\int_0^1x\frac{d}{dx}\left(\varphi_1^2(x)\right)dx=\frac{1}{2}\left[\left.x\varphi_1^2(x)\right|^1_0-\int_0^1\varphi_1^2(x)dx\right]=-\frac{1}{2},
\end{equation*}
where we used once again \cite[Proposition 2.5, Property (2.12)]{FABCL}) with $u=\varphi_1$. Thus, we have $\langle B\varphi_1,\varphi_1\rangle\neq0.$

Therefore, we proved that all the hypothesis of Theorem \eqref{teo-contr-B-unb} are satisfied. Applying such result, we deduce that
for any $T>0$ and initial condition $u_0\in D(A^{1/2})$ sufficiently close to the ground state $\varphi_1$, there exists a control $p\in L^2(0,T)$ that steers the solution of \eqref{eq-ex-deg-D} to $\psi_1(T,x)=e^{-\lambda_1T}\varphi_1(x)$ in time $T$, by the iterative constructive process we set up in our abstract results.

\subsection{Degenerate parabolic equation with Neumann boundary conditions}
In this section we study the controllability of the following degenerate problem
\begin{equation}\label{eq-ex-deg-N}
\left\{
\begin{array}{ll}
u_t-\left(x^{\alpha} u_x\right)_x+p(t)\mu(x)(u_x+u)=0,& (t,x)\in [0,T]\times I\\
(x^\alpha u_x)(t,0)=0,\quad u_x(t,1)=0, & t\in [0,T]\\
u(0,x)=u_0(x),&t\in x\in I.
\end{array}
\right.
\end{equation}
where $T>0$, $I=(0,1)$ and $\mu(x)=x^{2-\alpha}$. The control $p\in L^2(0,T)$ is a real valued function and $\mu$ represents an admissible potential.

Recalling the definitions of the weighted Sobolev spaces $H^1_\alpha(I)$ and $H^2_\alpha(I)$ in \eqref{sob-sp-w} and \eqref{sob-sp-s} for weak and strong degeneracy respectively, we define the second order linear operator
\begin{equation*}
\begin{cases}
\forall\, u \in D(A), \quad  Au:=- (x ^\alpha  u_x)_x , \\
D(A) :=  \{ u \in H^2_{\alpha} (0,1) , (x^\alpha u_x) (0)=0, u_x (1)=0 \} .
\end{cases} 
\end{equation*}
In \cite[Proposition 2.1, Proposition 2.2]{cmu} it is proved that for any $\alpha\in [0,2)$ the operator $A$ is self-adjoint accretive and has a dense domain. Therefore, $-A$ is the infinitesimal generator of an analytic semigroup of contraction. Moreover, in \cite[Proposition 3.1]{cmu} the authors showed that the eigenvalues and eigenfunctions of the operator $A$, when $\alpha\in [0,1)$, are given by
\begin{equation}\label{EXDEGeq-vp-w-vp0}
\lambda _{\alpha,0} = 0,\quad \varphi _{\alpha,0} (x)=1
\end{equation}
and for all $m\geq 1$
\begin{equation}\label{EXDEGeq-vp-w-vpm} 
\lambda _{\alpha,m} = \kappa _\alpha ^2 \, j_{-\nu_\alpha - 1, m} ^2 ,
\end{equation}
\begin{equation}\label{EXDEGeq-fp-w-fpm} 
\varphi _{\alpha,m} (x) = K_{\alpha,m}  x^{\frac{1-\alpha}{2}}  J_{-\nu_\alpha} \left( j_{-\nu_\alpha + 1, m}  x^{\frac{2-\alpha}{2}}\right), 
\end{equation}
where 
$$ \kappa _\alpha := \frac{2-\alpha}{2}, \quad \nu_\alpha := \frac{1-\alpha}{2-\alpha},$$
$J_{-\nu_\alpha}$ is the Bessel$^\prime$s function of order $-\nu_\alpha$, $(j_{-\nu_\alpha + 1, m}) _{m\geq 1}$ are the positive zeros of the Bessel$^\prime$s function $J_{-\nu_\alpha +1}$ and $K_{\alpha,m}$ are positive constants. 

For $\alpha\in[1,2)$, the eigenvalues and eigenfunctions of $A$ are defined by
\begin{equation}\label{EXDEGeq-vp-s-vp0}
\lambda _{\alpha,0} = 0,\quad \varphi _{\alpha,0} (x)=1
\end{equation}
and for all $m\geq 1$
\begin{equation}\label{EXDEGeq-vp-s-vpm} 
\lambda _{\alpha,m} = \kappa _\alpha ^2 \, j_{\nu_\alpha + 1, m} ^2 ,
\end{equation}
\begin{equation}\label{EXDEGeq-fp-s-fpm} 
\varphi _{\alpha,m} (x) = K_{\alpha,m}  x^{\frac{1-\alpha}{2}} J_{\nu_\alpha} \left( j_{\nu_\alpha + 1, m}\,  x^{\frac{2-\alpha}{2}}\right) ,
\end{equation}
where
$$ \kappa _\alpha := \frac{2-\alpha}{2}, \quad \nu_\alpha := \frac{\alpha -1}{2-\alpha},$$
$J_{\nu_\alpha}$ is the Bessel$^\prime$s function of order $\nu_\alpha$, $(j_{\nu_\alpha +1, m}) _{m\geq 1}$ are the positive zeros of the Bessel$^\prime$s function $J_{\nu_\alpha +1}$ and $K_{\alpha,m}$ are positive constants (see \cite[Proposition 3.2]{cmu}).

Observe that, in the same paper \cite[Propositions 3.1 and 3.2]{cmu} it is proved that the eigenvalues $\{\lambda_{\alpha,m}\}_{m\in\NN}$ satisfy the following gap condition
\begin{equation*}
\forall\, \alpha\in [0,2),\,\sqrt{\lambda_{\alpha,m+1}}-\sqrt{\lambda_{\alpha,m}}\geq\frac{2-\alpha}{2}\pi.
\end{equation*}

We introduce the operator $B:D(B)=:H^1(I)\subset X\to X$ defined by
\begin{equation*}
B\varphi=\mu\left(\varphi'+\varphi\right),\quad\forall\varphi\in D(B),
\end{equation*}
where $\mu(x)=x^{2-\alpha}$. In order to have that $D(A^{1/2})\hookrightarrow D(B)$ (that is, hypothesis \eqref{BA12}), we need to restrict the analysis to the case $\alpha\in[0,4/3]$. Indeed,
\begin{equation*}
\norm{B\varphi}=\norm{\mu(\varphi'+\varphi)}\leq C\left(\norm{x^{\alpha/2}\varphi'}+\norm{\varphi}\right)=C\left(\norm{A^{1/2}\varphi}+\norm{\varphi}\right)\leq C\norm{\varphi}_{D(A^{1/2})},\quad\forall \varphi\in D(A^{1/2})
\end{equation*}
is satisfied if and only if $\alpha\in[0,4/3]$.

Furthermore, it is possible to prove that also hypothesis \eqref{ipB-suff-cond} holds true. We compute the scalar product $\langle B\varphi_{\alpha,0},\varphi_{\alpha,m}\rangle$, for all $m\in\NN$:
$$ \langle B\varphi_{\alpha,0} , \varphi _{\alpha,0} \rangle = \int _0 ^1 x^{2-\alpha}  dx = \frac{1}{3-\alpha} ,$$
and, for all $m\geq 1$ we have that
\begin{equation*}
\begin{split}
\langle B\varphi_{\alpha,0} , \varphi _{\alpha,m}\rangle&=\int_0^1\mu(x)\left(\varphi_{\alpha,0}^\prime+\varphi_{\alpha,0}\right)\varphi_{\alpha,m}dx
= \int _0 ^1 \mu (x) \varphi _{\alpha,m}(x)  dx 
\\
&= \frac{1}{\lambda _{\alpha,m}} \int _0 ^1 \mu (x) (-x^\alpha \varphi _{\alpha,m}^\prime)^\prime(x)  dx
\\
&= \frac{1}{\lambda _{\alpha,m}}  \left( [-x^\alpha \mu (x) \varphi _{\alpha,m}^\prime(x) ] _0 ^1 + \int _0 ^1 x^\alpha \mu^\prime(x) \varphi _{\alpha,m}^\prime(x) \right).
\end{split}
\end{equation*}
Recalling that $\mu (x)=x^{2-\alpha}$, we obtain
\begin{equation*}
\begin{split}
\int _0 ^1 x^\alpha \mu^\prime(x) \varphi _{\alpha,m}^\prime (x) 
&= (2-\alpha) \int _0 ^1 x \varphi _{\alpha,m}^\prime (x) 
\\
&= (2-\alpha) \left[x \varphi _{\alpha,m}  (x)\right] _0 ^1 - (2-\alpha) \int _0 ^1 \varphi _{\alpha,m} (x)  dx
\\
&= (2-\alpha) \left[x \varphi _{\alpha,m}  (x)\right] _0 ^1 - (2-\alpha) \langle \varphi _{\alpha,0} , \varphi _{\alpha,m} \rangle  .
\end{split}
\end{equation*}
Since the eigenfunctions are orthogonal, we have that
$$ \langle \varphi _{\alpha,0} , \varphi _{\alpha,m} \rangle  = 0 ,$$
hence
$$ \langle B\varphi_{\alpha,0} , \varphi _{\alpha,m} \rangle 
= \frac{1}{\lambda _{\alpha,m}}  \left( \left[-x^2 \varphi _{\alpha,m}^\prime(x) \right] _0 ^1
+ (2-\alpha) \left[x \varphi _{\alpha,m}  (x)\right] _0 ^1 \right) .$$
From the Neumann boundary conditions satisfied by $\varphi _{\alpha,m}$, we know that
$x \varphi _{\alpha,m}^\prime(x) \to 0$ as $x\to 0$ and $x\to 1$, thus
$$  \left[-x^2 \varphi _{\alpha,m}^\prime(x) \right] _0 ^1 = 0 .$$
Furthermore, in \cite[Lemma 5.1 and Lemma 5.2]{cmu} it is shown that $\varphi _{\alpha,m}$ has a finite limit as $x\to 0$, therefore
$$ x \varphi _{\alpha,m}  (x) \to 0, \quad \text{ as } x \to 0,$$
and moreover  that $\vert \varphi _{\alpha,m} (1)\vert =\sqrt{2-\alpha}$, which implies
$$ \vert (2-\alpha) [x \varphi _{\alpha,m}  (x)] _0 ^1 \vert = (2-\alpha)^{3/2}.$$
Therefore,
$$ \vert \langle B\varphi_{\alpha,0} , \varphi _{\alpha,m} \rangle\vert = \frac{(2-\alpha)^{3/2}}{\lambda _{\alpha,m}} ,$$
that is condition \eqref{ipB-suff-cond} with $q=1$.

By applying Theorem \ref{teo-contr-B-unb}, we conclude that problem \eqref{eq-ex-deg-N}, for $\alpha\in[0,4/3]$, is exactly controllable to the ground state solution $\psi_0\equiv 1$ in any time $T>0$. 

Let us now show an application of Theorem \ref{teoglobal0} to Example 5.4. We recall that $\lambda_0=0$ is the first eigenvalue of $A$, and is associated to the eigenfunction $\varphi_0\equiv 1$. We set
\begin{equation}\label{exactphi0}
\phi_0(t)=\langle u_0,\varphi_1\rangle_{1/2}\; \psi_0(t)=\int_0^1 u_0(x)dx, \quad\forall\, t \geq 0.
\end{equation}
Theorem \ref{teoglobal0} applied to Example 5.4 gives:
\vskip 1mm

for any $R>0$ there exists $T_R>0$ such that for all $u_0\in D(A^{1/2})$ satisfying 
\begin{equation}\label{NV5.4}
\int_0^1 u_0(x)dx \neq 0,
\end{equation}
and 
\begin{equation}\label{cone5.4}
\left(\dfrac{\int_0^1u_0^2(x)dx -\big(\int_0^1u_0(x)dx\big)^2 + \int_0^1x^{\alpha}|u^{\prime}_{0}|^2(x)dx}
{\left|\int_0^1u_0(x)dx\right|}
\right)^{1/2} \leq R
\end{equation}
system \eqref{eq-ex-deg-N} is exactly controllable to $\phi_0$, defined in \eqref{exactphi0}, in time $T_R$.

\section*{Acknowledgement}
P. Cannarsa and C. Urbani acknowledge support from the MIUR Excellence Department Project awarded to the Department of Mathematics, University of Rome Tor Vergata, CUP E83C18000100006.

C. Urbani acknowledges support from Accademia Nazionale dei Lincei, Postdoc scholarship ``Beniamino Segre" and from Università Italo Francese (Vinci Project 2018).

\begin{appendices}
\section{}\label{AppendixB}
In this section we will prove the existence of strong solutions of the stochastic differential equation 
\begin{equation}\label{SDE-Appendix}
\begin{cases}
dX_t=\nu(t,X_t)dt+\sigma(t,X_t)dB_t\\
X(t=0)=X_0
\end{cases}
\end{equation}
under assumptions (H2) which are weaker than assumptions (H1) in \cite[Theorem 9.2]{baldi}. We develop the computations for a drift $\nu$ of the form 
$$
\nu(t,x)=p(t)\mu(x),\,\,t \in [0,T],\, x \in \RR .
$$
with $p\in L^2(0,T)$, which is the case of interest to this paper. A general drift $\nu$ satisfying (H2) can be treated in a similar way.

Let $(\Omega, \mathcal{F},\mathbb{P})$ be a probability space. We recall that the process $(X_t)_{t\in[0,T]}$ is a solution of \eqref{SDE-Appendix} if
\begin{enumerate}
\item $(\Omega, \mathcal{F},(\mathcal{F}_t)_{t\in[0,T]},(B_t)_{t\in[0,T]},\mathbb{P})$ is a standard Brownian motion
\item $(X_t)_{t\in[0,T]}$ is adapted to the filtration $(\mathcal{F}_t)_{t\in[0,T]}$
\item for every $t\in[0,T]$ :
\begin{equation*}
X_t=X_0+\int_0^tp(s)\mu(X_s)ds+\int_0^t \sigma(s,X_s)dB_s.
\end{equation*}
\end{enumerate}
We now prove the following preliminary result which adapts \cite[Theorem 9.1]{baldi} to the current setting.
\begin{thm}\label{teo91}
Let $X$ be a solution of
\begin{equation*}
X_t=X_0+\int_0^tp(s)\mu(X_s)ds+\int_0^t\sigma(s,X_s)dB_s
\end{equation*}
where $\mu,p$ and $\sigma$ satisfy hypotheses 1.,2. and 3. of (H2) (see page 20), and $X_0$ be a $\mathcal{F}_0$-measurable r.v. of $L^2$. Then,
\begin{equation}\label{1estim-app}
\mathbb{E}\left(\sup_{0\leq s\leq T}|X_s|^2\right)\leq c(T,M)(1+\mathbb{E}(|X_0|^2))
\end{equation}
\begin{equation}\label{2estim-app}
\mathbb{E}\left(\sup_{0\leq s\leq t}|X_s-X_0|^2\right)\leq c(T,M)t(1+\mathbb{E}(|X_0|^2)).
\end{equation}
\end{thm}
\begin{proof}
We follow the proof of \cite[Theorem 9.1]{baldi}. Fixed any $R>0$, we define $X_R(t):=X_{t\land \tau_R}$ where $\tau_R:=\inf\{t\,:\, 0\leq t\leq T,\, |X_t|\geq R\}$ is the exit time of the process $X$ from the open ball of radius $R$. If $|X_t|<R$ for every $t\in[0,T]$ then we set $\tau_R=T$. We have that
\begin{equation*}
\begin{split}
X_R(t)&=X_0+\int_0^{t\land \tau_R}p(r)\mu(X_r)dr+\int_0^{t\land \tau_R}\sigma(r,X_r)dB_r\\
&=X_0+\int_0^{t}p(r)\mu(X_r)\mathbbm{1}_{\{r<\tau_R\}}dr+\int_0^{t}\sigma(r,X_r)\mathbbm{1}_{\{r<\tau_R\}}dB_r\\
&=X_0+\int_0^{t}p(r)\mu(X_R(r))\mathbbm{1}_{\{r<\tau_R\}}dr+\int_0^{t}\sigma(r,X_R(r))\mathbbm{1}_{\{r<\tau_R\}}dB_r.
\end{split}
\end{equation*}
Therefore, we get
\begin{equation}\label{EsupX2}
\begin{split}
\mathbb{E}\left[\sup_{0\leq s\leq t}|X_R(s)|^2\right]&\leq 3\mathbb{E}\left[|X_0|^2\right]+3\mathbb{E}\left[\sup_{0\leq s\leq t}\left|\int_0^sp(r)\mu(X_R(r))\mathbbm{1}_{\{r<\tau_R\}}rd\right|^2\right]\\
&\quad+3\mathbb{E}\left[\sup_{0\leq s\leq t}\left|\int_0^s\sigma(r,X_R(r))\mathbbm{1}_{\{r<\tau_R\}}\right|^2\right],
\end{split}
\end{equation}
where we have used the inequality
\begin{equation*}
|x_1+\dots+x_m|^p\leq m^{p-1}(|x_1|^p+\dots+|x_m|^p)
\end{equation*}
which holds, in general, for any $x_1,\dots,x_m\in\RR^d$.

By H{\"o}lder's inequality and hypothesis 3. we obtain
\begin{equation}\label{Esuppmu2}
\begin{split}
\mathbb{E}\left[\sup_{0\leq s\leq t}\left|\int_0^sp(r)\mu(X_R(r))\mathbbm{1}_{\{r<\tau_R\}}dr\right|^2\right]&\leq \mathbb{E}\left[\sup_{0\leq s\leq t}\left(\left(\int_0^s|p(r)|^2dr\right)\left(\int_0^s|\mu(X_R(r))|^2\mathbbm{1}_{\{r<\tau_R\}}rd\right)\right)\right]\\
&\leq \norm{p}^2_{L^2(0,T)}M^2\mathbb{E}\left[\sup_{0\leq s\leq t}\int_0^s(1+|X_R(r)|)^2\mathbbm{1}_{\{r<\tau_R\}}dr\right].
\end{split}
\end{equation}
Whereas, by the $L^2$ inequalities for stochastic integrals (see \cite[Proposition 8.4]{baldi}) we deduce 
\begin{equation}\label{Esupsigma2}
\mathbb{E}\left[\sup_{0\leq s\leq t}\left|\int_0^s\sigma(r,X_R(r))\mathbbm{1}_{\{r<\tau_R\}}\right|^2\right]\leq c_2\mathbb{E}\left[\int_0^t|\sigma(r,X_R(r))|^2dr\right]\leq c_2M^2\mathbb{E}\left[\int_0^t\left(1+|X_R(r)|\right)^2dr\right],
\end{equation}
where in the last inequality we have used again hypothesis 3. Now, plugging \eqref{Esuppmu2} and \eqref{Esupsigma2} into \eqref{EsupX2}, we get
\begin{equation*}
\begin{split}
\mathbb{E}\left[\sup_{0\leq s\leq t}|X_R(s)|^2\right]&\leq 3\mathbb{E}\left[|X_0|^2\right]+3M^2\left(\norm{p}^2_{L^2(0,T)}+c_2\right)2\mathbb{E}\left[T^2+\int_0^t|X_R(r)|^2dr\right]\\
&\leq c_1(T,M)\left(1+\mathbb{E}\left[|X_0|^2\right]\right)+c_2(T,M)\int_0^t\mathbb{E}\left[|X_R(r)|^2\right]dr.
\end{split}
\end{equation*}
If we set $v(t):=\mathbb{E}\left[\sup_{0\leq s\leq t}|X_R(s)|^2\right]$, from the above estimate we obtain
\begin{equation*}
v(t)\leq c_1(T,M)\left(1+\mathbb{E}\left[|X_0|^2\right]\right)+c_2(T,M)\int_0^tv(r)dr.
\end{equation*}
Observe that, since $|X_R(t)|=|X_0|$ if $|X_0|>R$ and $|X_R(t)|\leq R$ otherwise, it holds that $|X_R(t)|\leq \max\{|X_0|,R\}$ and then $v(t)\leq \mathbb{E}\left[\max\{|X_0|^2,R^2\}\right]<+\infty$. Thus, $v$ is bounded and we can apply Gronwall's inequality:
\begin{equation*}
v(T)=\mathbb{E}\left[\sup_{0\leq s\leq T}|X_R(s)|^2\right]\leq c_1(T,M)\left(1+\mathbb{E}\left[|X_0|^2\right]\right)e^{Tc_2(T,M)}=c(T,M)\left(1+\mathbb{E}\left[|X_0|^2\right]\right).
\end{equation*}
Since the right-hand side does not depend on $R$, we can take the limit $R\to+\infty$.

Let us prove that $\tau_R\to T$ as $R\to+\infty$. Since $X$ is continuous, $\displaystyle\sup_{0\leq t\leq \tau_R}|X_t|^2=R^2$ on $\{\tau_R<T\}$. Therefore, we have that
\begin{equation*}
\mathbb{E}\left[\sup_{0\leq t\leq \tau_R}|X_t|^2\right]\geq R^2\mathbb{P}(\tau_R<T)
\end{equation*}
so that
\begin{equation*}
\mathbb{P}(\tau_R<T)\leq \frac{1}{R^2}\mathbb{E}\left[\sup_{0\leq t\leq \tau_R}|X_t|^2\right]\leq \frac{c(T,M)(1+\mathbb{E}(|X_0|^2)}{R^2}.
\end{equation*}
Hence, $\mathbb{P}(\tau_R<T)\to 0$ as $R\to+\infty$. Since $R\to \tau_R$ is increasing, $\displaystyle\lim_{R\to +\infty}\tau_R=T$ a.s. and 
\begin{equation*}
\sup_{0\leq s\leq T}|X_R(s)|^2\to\sup_{0\leq s\leq T}|X_s|^2\quad{a.s.,}\quad{for}\quad R\to+\infty.
\end{equation*}
By applying Fatou's lemma we conclude that \eqref{1estim-app} holds true. We refer to \cite[Theorem 9.1]{baldi} for the proof of \eqref{1estim-app}.
\end{proof}
Let $a,b\in\RR$ such that $a\leq b$. We recall the definition of the spaces $M^p([a,b])$, with $p\geq1$. First, we define the space $M^p_{loc}([a,b])$ as the space of the equivalence classes of real-valued \emph{progressively measurable} processes $X=(\Omega, \mathcal{F},(\mathcal{F}_t)_{a\leq t\leq b},(X_t)_{a\leq t\leq b},\mathbb{P})$ such that
\begin{equation*}
\int_a^b|X_s|^pds<+\infty\qquad \text{a.s.}
\end{equation*}
Then, by $M^p([a,b])$ we denote the subspace of $M^p_{loc}([a,b])$ of the processes such that
\begin{equation*}
\mathbb{E}\left[\int_a^b|X_s|^pds\right]<+\infty.
\end{equation*}
\begin{thm}
Let $X_0$ be a real-valued r.v., $\mathcal{F}_0$ measurable and square integrable. Then, under assumption (H2) (see page 20) there exists $X\in M^2([0,T])$ that satisfies
\begin{equation}\label{sol-sde}
X_t=X_0+\int_0^tp(t)\mu(X_s)ds+\int_0^t\sigma(s,X_s)dB_s.
\end{equation}
Moreover, if $X'$ is another solution of \eqref{sol-sde}, then
\begin{equation}
\mathbb{P}(X_t=X'_t\text{ for every }t\in[0,T])=1
\end{equation}
(pathwise uniqueness).
\end{thm}
\begin{proof}
We follow the proof of \cite[Theorem 9.2]{baldi}. We define recursively a sequence of processes by $X_0(t)=X_0$ and 
\begin{equation*}
X_{m+1}(t)=X_0+\int_0^tp(s)\mu(X_m(s))ds+\int_0^t\sigma(s,X_m(s))dB_s.
\end{equation*}
Our aim is to show that the sequence $(X_m)_m$ converges uniformly on $[0,T]$ to a process $X$ which is solution of \eqref{sol-sde}.

We first prove by induction that
\begin{equation}\label{induction-appA}
\mathbb{E}\left[\sup_{0\leq r\leq t}|X_{m+1}(r)-X_m(r)|^2\right]\leq \frac{(Rt)^{m+1}}{(m+1)!}
\end{equation}
with $R:=2(\norm{p}^2_{L^2(0,T)}+4)\max\{2M^2(1+\mathbb{E}[|X_0|^2]),L^2\}$, where $L,M>0$ are the constants in hypotheses 3. and 4..

For $m=0$ we have that
\begin{equation*}
\begin{split}
\sup_{0\leq r\leq t}|X_1(r)-X_0|^2&\leq 2\sup_{0\leq r\leq t}\left|\int_0^rp(s)\mu(X_0)ds\right|^2+2\sup_{0\leq r\leq t}\left|\int_0^r\sigma(s,X_0)dB_s\right|^2\\
&\leq 2\norm{p}^2_{L^2(0,T)}\sup_{0\leq r\leq t}\int_0^r\left|\mu(X_0)\right|^2ds+2\sup_{0\leq r\leq t}\left|\int_0^r\sigma(s,X_0)dB_s\right|^2\\
&\leq 4M^2\norm{p}^2_{L^2(0,T)}\left(1+|X_0|\right)^2t+2\sup_{0\leq r\leq t}\left|\int_0^r\sigma(s,X_0)dB_s\right|^2
\end{split}
\end{equation*}
where we have applied H{\"o}lder's inequality and hypothesis 3..

By using Doob's maximal inequality (see \cite[formula (7.23), page 195]{baldi}) and hypothesis 3. of (H2) it is possible to prove that
\begin{equation*}
\mathbb{E}\left[\sup_{0\leq r\leq t}\left|\int_0^r\sigma(s,X_s)dB_s\right|^2\right]\leq 4\mathbb{E}\left[\int_0^t|\sigma(s,X_s)|^2ds\right]\leq 8M^2\left(t+\mathbb{E}\left[\int_0^t|X_s|^2ds\right]\right).
\end{equation*}
Hence, we get that
\begin{equation*}
\begin{split}
\mathbb{E}\left[\sup_{0\leq r\leq t}|X_1(r)-X_0|^2\right]&\leq 4M^2\norm{p}_{L^2(0,T)}^2\left(1+\mathbb{E}\left[|X_0|^2\right]\right)t+16M^2\left(t+t\mathbb{E}\left[|X_0|^2\right]\right)\\
&=4M^2(\norm{p}^2_{L^2(0,T)}+4)\left(1+\mathbb{E}\left[|X_0|^2\right]\right)t \leq Rt.
\end{split}
\end{equation*}
Now, suppose that \eqref{induction-appA} holds till index $m-1$ and let us prove it for $m$. Observe that, thanks to H{\"o}lder's inequality and hypothesis 4., it holds
\begin{equation*}
\begin{split}
\sup_{0\leq r\leq t}|X_{m+1}(r)-X_m(r)|^2&\leq 2\sup_{0\leq r\leq t}\left|\int_0^rp(s)\left[\mu(X_{m}(s))-\mu(X_{m-1}(s))\right]ds\right|^2\\
&\quad+2\sup_{0\leq r\leq t}\left|\int_0^r\left[\sigma(s,X_{m}(s))-\sigma(s,X_{m-1}(s))\right]dB_s\right|^2\\
&\leq 2\norm{p}_{L^2(0,T)}\sup_{0\leq r\leq t}\int_0^r\left|\mu(X_{m}(s))-\mu(X_{m-1}(s))\right|^2ds\\
&\quad+2\sup_{0\leq r\leq t}\left|\int_0^r\left[\sigma(s,X_{m}(s))-\sigma(s,X_{m-1}(s))\right]dB_s\right|^2\\
&\leq 2L^2\norm{p}_{L^2(0,T)}\sup_{0\leq r\leq t}\int_0^r\left|X_{m}(s)-X_{m-1}(s)\right|^2ds\\
&\quad+2\sup_{0\leq r\leq t}\left|\int_0^r\left[\sigma(s,X_{m}(s))-\sigma(s,X_{m-1}(s))\right]dB_s\right|^2.
\end{split}
\end{equation*}
We now compute the expected value
\begin{equation*}
\begin{split}
\mathbb{E}\left[\sup_{0\leq r\leq t}|X_{m+1}(r)-X_m(r)|^2\right]&\leq 2L^2\norm{p}_{L^2(0,T)}\int_0^t\mathbb{E}\left[\left|X_{m}(s)-X_{m-1}(s)\right|^2\right]ds\\
&\quad+2\mathbb{E}\left[\sup_{0\leq r\leq t}\left|\int_0^r\left[\sigma(s,X_{m}(s))-\sigma(s,X_{m-1}(s))\right]dB_s\right|^2\right].
\end{split}
\end{equation*}
By using again Dobb's inequality, hypothesis 4. and the inductive step, we have
\begin{equation*}
\begin{split}
\mathbb{E}\left[\sup_{0\leq r\leq t}|X_{m+1}(r)-X_m(r)|^2\right]&\leq 2L^2\norm{p}_{L^2(0,T)}\int_0^t\mathbb{E}\left[\left|X_{m}(s)-X_{m-1}(s)\right|^2\right]ds\\
&\quad+8\mathbb{E}\left[\int_0^t\left|\sigma(s,X_{m}(s))-\sigma(s,X_{m-1}(s))\right|^2ds\right]\\
&\leq 2L^2\norm{p}_{L^2(0,T)}\int_0^t\mathbb{E}\left[\left|X_{m}(s)-X_{m-1}(s)\right|^2\right]ds\\
&\quad+8L^2\int_0^t\mathbb{E}\left[\left|X_{m}(s)-X_{m-1}(s)\right|^2\right]ds\\
&\leq 2L^2\left(\norm{p}^2_{L^2(0,T)}+4\right)\int_0^t\frac{(Rs)^m}{m!}ds\leq \frac{(Rt)^{m+1}}{(m+1)!}.
\end{split}
\end{equation*}
Thus, the proof of \eqref{induction-appA} is completed.

Now we apply Markov's inequality which gives
\begin{equation*}
\mathbb{P}\left(\sup_{0\leq t\leq T}|X_{m+1}(t)-X_m(t)|>\frac{1}{2^m}\right)\leq 2^{2m}\mathbb{E}\left[\sup_{0\leq t\leq T}|X_{m+1}(t)-X_m(t)|^2\right]\leq 2^{2m}\frac{(RT)^{m+1}}{(m+1)!}.
\end{equation*}
Since the left-hand side of the above inequality is summable, by the Borel-Cantelli lemma we get
\begin{equation*}
\mathbb{P}\left(\sup_{0\leq t\leq T}|X_{m+1}(t)-X_m(t)|>\frac{1}{2^m}\text{ for infinitely many indices $m$}\right)=0,
\end{equation*}
that is, for almost every $\omega$ we eventually have
\begin{equation*}
\sup_{0\leq t\leq T}|X_{m+1}(t)-X_m(t)|\leq \frac{1}{2^m}
\end{equation*}
and therefore, for fixed $\omega$ the series
\begin{equation*}
X_0+\sum_{k=0}^{m-1}|X_{k+1}(t)-X_k(t)|=X_m(t)
\end{equation*}
converges uniformly on $[0,T]$ a.s. Let $X_t=\displaystyle\lim_{m\to+\infty} X_m(t)$. Then, $X$ is continuous, being the uniform limit of continuous processes, and therefore $X\in M^2_{loc}([0,T])$.

We now prove that $X$ is solution of \eqref{sol-sde}. Recall that
\begin{equation}\label{sol-sde2}
X_{m+1}(t)=X_0+\int_0^tp(s)\mu(X_m(s))ds+\int_0^t\sigma(s,X_m(s))dB_s.
\end{equation}
Then, the left-hand side converges uniformly to $X$. From hypothesis 3. of (H2) we get that
\begin{equation*}
|p(t)||\mu(X_{m}(t))-\mu(X_t)|\leq L|p(t)|\left|X_{m}(t)-X_t\right|.
\end{equation*}
By Lebesgue's dominated convergence theorem, we deduce that
\begin{equation*}
\int_0^tp(s)\mu(X_m(s))ds\to\int_0^t p(s)\mu(X_s)ds\quad\text{a.s.}\quad\quad\text{as}\quad m\to+\infty.
\end{equation*}
Moreover, since by hypothesis 3. of (H2)
\begin{equation*}
|\sigma(s,X_m(s))-\sigma(s,X_s)|\leq L|X_m(s)-X_s|
\end{equation*}
we deduce that, uniformly on $[0,T]$ a.s.,
\begin{equation*}
\lim_{m\to+\infty} \sigma(t,X_m(t))=\lim_{m\to+\infty}\sigma(t,X_m)
\end{equation*}
and therefore
\begin{equation*}
\int_0^t\sigma(s,X_m(s))ds\to\int_0^t \sigma(X_s)ds\quad\text{a.s.}\quad\text{as}\quad m\to+\infty.
\end{equation*}
Since a.s. convergence implies convergence in probability, we can take the limit in probability in \eqref{sol-sde2} and obtain that
\begin{equation*}
X_t=X_0+\int_0^tp(s)\mu(X_s)ds+\int_0^t\sigma(s,X_s)dB_s,
\end{equation*}
that is, $X$ is solution of the stochastic differential equation. Moreover, $X\in M^2([0,T])$ thanks to Theorem \ref{teo91}.

We now prove uniqueness. Let $X_1$, $X_2$ be two solutions of \eqref{sol-sde}. Then, we have
\begin{equation*}
|X_1(t)-X_2(t)|\leq \left|\int_0^tp(s)\left(\mu(X_1(s))-\mu(X_2(s)\right)ds\right|+\left|\int_0^t\left(\sigma(s,X_1(s))-\sigma(s,X_2(s))\right)dB_s\right|.
\end{equation*}
We compute the mean value of the supremum over $[0,t]$ of the difference of the solutions
\begin{equation*}
\begin{split}
\mathbb{E}\left[\sup_{0\leq r\leq t}\left|X_1(r)-X_2(r)\right|^2\right]&\leq 2\mathbb{E}\left[\sup_{0\leq r\leq t}\left|\int_0^rp(s)\left(\mu(X_1(s))-\mu(X_2(s)\right)ds\right|^2\right]\\
&\quad+2\mathbb{E}\left[\sup_{0\leq r\leq t}\left|\int_0^t\left(\sigma(s,X_1(s))-\sigma(s,X_2(s))\right)dB_s\right|^2\right]\\
&\leq 2\norm{p}^2_{L^2(0,T)}\mathbb{E}\left[\int_0^t\left|\mu(X_1(s))-\mu(X_2(s)\right|^2ds\right]\\
&\quad +8\mathbb{E}\left[\int_0^t\left|\sigma(s,X_1(s))-\sigma(s,X_2(s))\right|^2ds\right]\\
&\leq \left(2L^2\norm{p}_{L^2(0,T)}^2+8L^2\right)\int_0^t\mathbb{E}\left[\left|X_1(s)-X_2(s)\right|^2\right]ds.
\end{split}
\end{equation*}
If we set $v(t)=\sup_{0\leq r\leq t}\left|X_1(r)-X_2(r)\right|^2$, then $v$ is bounded thanks to Theorem \ref{teo91} and it from the above inequality we deduce 
\begin{equation*}
v(t)\leq C\int_0^t v(s)ds,\qquad 0\leq t\leq T
\end{equation*}
with $C:=2L^2\left(\norm{p}_{L^2(0,T)}^2+4\right)$. From Gronwall's inequality we conclude that $v\equiv 0$ on $[0,T]$. Thus, the two solutions $X_1$ and $X_2$ coincide.
\end{proof}
\end{appendices}

\bibliographystyle{plain}
\bibliography{biblio}
\end{document}